\documentclass[pdflatex,sn-mathphys-num]{sn-jnl}

\usepackage{graphicx}%
\usepackage{multirow}%
\usepackage{amsmath,amssymb,amsfonts}%
\usepackage{mathtools}%
\usepackage{amsthm}%
\usepackage{mathrsfs}%
\usepackage[title]{appendix}%
\usepackage{xcolor}%
\usepackage{textcomp}%
\usepackage{manyfoot}%
\usepackage{booktabs}%
\usepackage{enumitem}%

\numberwithin{equation}{section}

\theoremstyle{thmstyleone}%
\newtheorem{theorem}{Theorem}[section]
\newtheorem{lemma}[theorem]{Lemma}
\newtheorem{corollary}[theorem]{Corollary}
\newtheorem{proposition}[theorem]{Proposition}

\newtheorem{externaltheoreminner}{Theorem}
\newenvironment{externaltheorem}[1]{%
\renewcommand{\theexternaltheoreminner}{#1}%
\begin{externaltheoreminner}%
}{%
\end{externaltheoreminner}%
}

\theoremstyle{thmstylethree}%

\theoremstyle{thmstyletwo}%
\newtheorem{remark}[theorem]{Remark}

\newcommand{\Rn}{\mathbb R^n}
\newcommand{\R}{\mathbb R}

\newcommand{\Snn}{\mathbb S^{n-1}}

\makeatletter
\newcommand{\norm}[1]{\left\lVert #1\right\rVert\@ifnextchar\bgroup{\norm@sub}{}}
\newcommand{\norm@sub}[1]{_{#1}}
\makeatother
\newcommand{\abs}[1]{\left\lvert #1\right\rvert}

\begin{document}

\title[Extremizers for a trilinear Stein–Weiss inequality with nonnegative weights]{Extremizers for a trilinear Stein–Weiss inequality with nonnegative weights}

\author[1]{\fnm{Chengcheng} \sur{Wu}}\email{wuchengcheng@amss.ac.cn}
\author[2]{\fnm{Ziyu} \sur{Gan}}\email{24B912023@stu.hit.edu.cn}
\author*[3]{\fnm{Yongliang} \sur{Zhou}}\email{zhouyongliang@hrbeu.edu.cn}

\affil[1]{\orgdiv{School of Mathematical Sciences}, \orgname{Shanxi University}, \orgaddress{\city{Taiyuan}, \country{P.R. China}}}

\affil[2]{\orgdiv{School of Mathematics}, \orgname{Harbin Institute of Technology}, \orgaddress{\city{Harbin}, \country{P.R. China}}}

\affil*[3]{\orgdiv{College of Mathematical Sciences}, \orgname{Harbin Engineering University}, \orgaddress{\city{Harbin}, \country{P.R. China}}}

\abstract{
	We study extremizers for a trilinear Stein-Weiss inequality on $\mathbb{R}^n$. Within the known boundedness region, we prove attainment under two additional assumptions: all six weight exponents are nonnegative, and at least one pair of Lebesgue exponents is admissible.
	The proof combines symmetric decreasing rearrangement with a logarithmic radial reduction to a translation-invariant bilinear operator on $\mathbb{R}$ whose kernel belongs to $L^1\left(\mathbb{R}^2\right)$. A common-scale compactness argument rules out relative separation of the two arguments and yields norm attainment. We then derive the Euler-Lagrange system. In the fully symmetric case, every normalized nonnegative extremizing triple is diagonal. Finally, we establish the origin-centered Kelvin invariance of the resulting scalar equation at the scaling exponent and record the unweighted conformal example.
	}

\keywords{Trilinear Stein--Weiss inequality, extremizers, weighted fractional integrals, Euler--Lagrange equation}

\pacs[MSC Classification]{42B20, 26D15, 46E30, 45G05}

\maketitle
\section{Introduction and main results}\label{sec:introduction}

In this paper, we study the attainability of the best constant for the trilinear Stein–Weiss inequality
\begin{equation}\label{ineq1}
	|\Lambda_A(f_1,f_2,f_3)|\le C\prod_{i=1}^3 \|f_i\|_{L^{p_i}(\Rn)}.
\end{equation}
Here \(\Lambda_A(f_1,f_2,f_3)\) denotes the trilinear Stein–Weiss functional
\begin{equation}\label{eq:Lambda-a}
	\Lambda_A(f_1,f_2,f_3):=
	\int_{\mathbb{R}^n} \int_{\mathbb{R}^n} \int_{\mathbb{R}^n} \prod_{1 \leqslant i<j \leqslant 3}\left|x_i-x_j\right|^{-\alpha_{i j}} \prod_{i=1}^3 f_i\left(x_i\right)\left|x_i\right|^{-\alpha_{i }} d x_1 d x_2 d x_3.
\end{equation}
The problem of the attainability of the best constant for \eqref{ineq1} asks whether the supremum
\[
C_A(p_1,p_2,p_3):=\sup_{\|f_i\|_{L^{p_i}(\Rn)}=1,\, i=1,2,3}|\Lambda_A(f_1,f_2,f_3)|
\]
is attained by a triple \((f_1,f_2,f_3)\) of unit norm.

The inequality \eqref{ineq1} is a natural trilinear extension of the classical Stein–Weiss inequality
\[
\biggl|\int_{\mathbb{R}^n} \int_{\mathbb{R}^n} \frac{f(x)\,g(y)}{|x|^\alpha|x-y|^\lambda|y|^\beta}\,dx\,dy\biggr| \leqslant C\,\|f\|_{L^p(\mathbb{R}^n)}\|g\|_{L^q(\mathbb{R}^n)}.
\]
The Stein--Weiss inequality, often viewed as the doubly weighted Hardy--Littlewood--Sobolev inequality, is the basic power-weighted estimate for the Riesz potential \cite{SteinWeiss1958}. It has served as a model for the development of the one- and two-weight theory of fractional integrals, from the foundational criteria of Muckenhoupt--Wheeden and Sawyer to sharp quantitative and multilinear extensions \cite{MuckenhouptWheeden1974,Sawyer1988,SawyerWheeden1992,LaceyMoenPerezTorres2010,LiMoenSun2015}. Recent work has continued to develop upper-half-space and endpoint variants of Stein--Weiss type inequalities \cite{LiShenSquassinaYang2024,SunWang2026}.
In the unweighted Hardy--Littlewood--Sobolev case, Lieb \cite{Lieb1983} determined the sharp constant and the full conformal family of extremizers. Related compactness and symmetry methods include rearrangement inequalities, concentration--compactness, competing symmetries, and inversion positivity \cite{BrascampLiebLuttinger1974,Lions1984I,Lions1984II,AlmgrenLieb1989,FrankLieb2010}. For the Stein--Weiss inequality with nonnegative power-weight exponents, Lieb \cite{Lieb1983} also proved the existence of extremizers. Chen, Lu, and Tao later extended the Euclidean existence theory beyond the regime in which both power-weight exponents are nonnegative, using concentration--compactness, and established corresponding results on the Heisenberg group \cite{CLT}. For a survey of geometric inequalities in harmonic analysis, including fractional integral inequalities and related extremal problems, see \cite{chen2018geometric} and the references therein.

Multilinear fractional integrals and their weighted variants have likewise attracted sustained attention. The Brascamp--Lieb inequalities provide a powerful framework for multilinear forms generated by linear maps, particularly in Gaussian settings \cite{BrascampLieb1976,Lieb1990,BennettCarberyChristTao2008}. Shi, Wu, and Yan established boundedness characterizations and endpoint estimates for multilinear fractional integral operators with correlation kernels \cite{ShiWuYan2019}. For general multilinear fractional integral inequalities with power weights, also referred to as the $k$-linear Stein--Weiss inequality, Zhou et al. obtained necessary conditions, sufficient conditions, and a complete characterization in a distinguished case \cite{ZhouDengWuYan2022}. Recently, a complete characterization of the inequality \eqref{ineq1} was obtained in \cite{ZhouXuWu2026}.

Once the boundedness region is known, three natural variational questions arise: Is the best constant attained? Can the extremizers be characterized? Can the best constant be evaluated explicitly? For the Stein--Weiss inequality, sharp constants are available in several boundary or special parameter regimes \cite{Beckner2008,WuShiYan2014}, while the existence of extremizers in the interior is known from \cite{Lieb1983,CLT}. These results motivate the corresponding questions for the inequality \eqref{ineq1}.

For inequality \eqref{ineq1}, Wu and Yan obtained an integral formula for the best constant in the boundary case --- $\sum_{i=1}^{3}1/p_i=1$ \cite{wu2016sharp}. 
We therefore begin with the interior ($\sum_{i=1}^{3}1/p_i > 1$) attainment problem. 
The full problem remains difficult, and the present paper establishes attainment in the subregion characterized by nonnegative
weights and the pairwise condition. More precisely, by combining the known boundedness theorem with Lieb's framework for Stein--Weiss extremizers, we prove attainment under two additional assumptions: all power-weight exponents are nonnegative, and at least one pair of Lebesgue exponents is admissible, in the sense specified below. We then derive the Euler--Lagrange system for extremizers. In the fully symmetric case, we prove that every normalized nonnegative extremizing triple is diagonal, that is, its three components agree almost everywhere. At the scaling exponent, we also establish the Kelvin invariance of the resulting scalar equation.

To state our results, We first recall the boundedness theorem for functional \eqref{eq:Lambda-a} in the form used below.
\begin{externaltheorem}{A}[Theorem 1, \cite{ZhouXuWu2026}]
\label{thm:boundedness}
Assume \(1<p_i<\infty\), \(i=1,2,3\), and \(p_i'=p_i/(p_i-1)\).  Then the inequality \eqref{ineq1} holds for all \(f_i\in L^{p_i}(\Rn)\) if and only if
\begin{gather}
\frac1{p_1}+\frac1{p_2}+\frac1{p_3}
+\frac{\alpha_1+\alpha_2+\alpha_3+\alpha_{12}+\alpha_{13}+\alpha_{23}}{n}=3,\label{eq:scale}\\
\alpha_1+\alpha_2+\alpha_3\ge0,\label{eq:sumone}\\
\alpha_{ij}<n\quad (i\ne j),\label{eq:pairlt}\\
\alpha_i<\frac{n}{p_i'}\quad (i=1,2,3),\label{eq:onelt}\\
\alpha_i+\sum_{j\ne i}\alpha_{ij}>\frac{n}{p_i'}\quad (i=1,2,3),\label{eq:tailgt}\\
\alpha_{12}+\alpha_{13}+\alpha_{23}<2n,\label{eq:pairsum}\\
\frac1{p_1}+\frac1{p_2}+\frac1{p_3}\ge1.\label{eq:pisum}
\end{gather}
\end{externaltheorem}

The functional \eqref{eq:Lambda-a} contains both one-body and pairwise singular weights. The scale condition \eqref{eq:scale} always leaves a common dilation symmetry. Depending on the point weights, translation invariance may be broken, but dilation remains a serious source of noncompactness. A maximizing sequence may concentrate at a common scale, drift to a different scale, or split so that different components occupy separated logarithmic scales. The pairwise factors create additional singular regimes near partial diagonals. Consequently, compactness does not follow directly from the existing Stein--Weiss theory \cite{Lieb1983,CLT} or from the usual conformal argument for the Hardy--Littlewood--Sobolev inequality.

We treat the attainment problem for the best constant of inequality \eqref{ineq1} through the associated bilinear operator
\begin{equation}\label{eq:def-TA-intro}
T_A(f_1,f_2)(x_3):=
\iint_{\Rn\times\Rn}\frac{f_1(x_1)|x_1|^{-\alpha_1}f_2(x_2)|x_2|^{-\alpha_2}|x_3|^{-\alpha_3}}{
	|x_1-x_2|^{\alpha_{12}}|x_1-x_3|^{\alpha_{13}}|x_2-x_3|^{\alpha_{23}}}\,dx_1dx_2.
\end{equation}
Initially, for bounded compactly supported functions $f_1$ and $f_2$, we define $T_A\left(f_1, f_2\right)$ by the integral above whenever it is finite. The trilinear estimate in Theorem A implies that, for every $\left(f_1, f_2\right) \in L^{p_1}\left(\mathbb{R}^n\right) \times L^{p_2}\left(\mathbb{R}^n\right)$, there exists a unique element $T_A\left(f_1, f_2\right) \in L^{p_3^{\prime}}\left(\mathbb{R}^n\right)$ such that

$$
\int_{\mathbb{R}^n} T_A\left(f_1, f_2\right)(x) h(x) d x=\Lambda_A\left(f_1, f_2, h\right) \quad \text { for all } h \in L^{p_3}\left(\mathbb{R}^n\right)
$$
This defines a bounded bilinear extension. For nonnegative inputs, truncation of the kernel, Tonelli's theorem, and monotone convergence show that this $L^{p_3^{\prime}}$ representative agrees almost everywhere with the extended pointwise integral. If $f_1$ and $f_2$ are radial, rotation invariance of the kernel implies that $T_A\left(f_1, f_2\right)$ is radial as an $L^{p_3^{\prime}}$ function.

We set
\begin{equation}\label{eq:norm-TA-intro}
\norm{T_A}:=
\sup_{0\ne f_i\in L^{p_i},\, i=1,2}
\frac{\norm{T_A(f_1,f_2)}{L^{p_3'}(\Rn)}}{\norm{f_1}{L^{p_1}(\Rn)}\norm{f_2}{L^{p_2}(\Rn)}}.
\end{equation}
By H\"older duality, \(\norm{T_A}=C_A(p_1,p_2,p_3)\).  Hence norm attainment for \(T_A\) yields attainment of the trilinear best constant.
In addition to \eqref{eq:scale}--\eqref{eq:pisum}, we impose two structural assumptions.  First, all weight exponents are nonnegative:
\begin{equation}\label{eq:nonneg}
\alpha_i\ge0\quad (i=1,2,3),\quad
\alpha_{ij}\ge0\quad (1\le i<j\le3).
\end{equation}
This hypothesis is not part of the general boundedness criterion. It is used both in the Brascamp--Lieb--Luttinger rearrangement step and in proving that the reduced kernel in logarithmic variables belongs to \(L^1(\R^2)\). Second, we assume that at least one pair of Lebesgue exponents is admissible in the following sense:
\begin{equation}\label{eq:pairwise}
\max\left\{\frac1{p_1}+\frac1{p_2},\frac1{p_1}+\frac1{p_3},\frac1{p_2}+\frac1{p_3}\right\}\ge1.
\end{equation}
After relabeling the variables, this will be written as
\begin{equation}\label{eq:pair12}
\frac1{p_1}+\frac1{p_2}\ge1.
\end{equation}
In the present proof, this condition is used only in the compactness argument; it prevents the two functions of the bilinear problem from separating at different logarithmic scales.

Our first result is norm attainment for \(T_A\).
\begin{theorem}\label{thm:TA-attainment}
Let \(1<p_i<\infty\) and \(p_i'=p_i/(p_i-1)\) for \(i=1,2,3\). 
Assume that \eqref{eq:scale}--\eqref{eq:pisum}, \eqref{eq:nonneg}, and \eqref{eq:pair12} hold. Then
\[
T_A:L^{p_1}(\Rn)\times L^{p_2}(\Rn)\to L^{p_3'}(\Rn)
\]
attains its norm.  More precisely, there exist nonnegative, radial, and radially nonincreasing functions \(f_1\in L^{p_1}(\Rn)\) and \(f_2\in L^{p_2}(\Rn)\) such that $
\norm{f_1}{L^{p_1}(\Rn)}=1$, $
\norm{f_2}{L^{p_2}(\Rn)}=1$, and 
\[
\norm{T_A(f_1,f_2)}{L^{p_3'}(\Rn)}=
\norm{T_A}.
\]
\end{theorem}

The trilinear extremizer follows by duality.
\begin{corollary}\label{cor:trilinear-attainment}
Let \(1<p_i<\infty\) and \(p_i'=p_i/(p_i-1)\) for \(i=1,2,3\). 
Assume that \eqref{eq:scale}--\eqref{eq:pisum}, \eqref{eq:nonneg}, and
\eqref{eq:pairwise} hold. Then there exist nonnegative, radial, and radially
nonincreasing functions \(f_i\in L^{p_i}(\Rn)\), \(i=1,2,3\), such that $
\norm{f_i}{L^{p_i}(\Rn)}=1$ for each $i$
and
\[
\Lambda_A(f_1,f_2,f_3)=C_A(p_1,p_2,p_3).
\]
\end{corollary}

We next record the variational structure of extremizers. Proposition~\ref{prop:EL} gives the Euler--Lagrange system.  In the fully symmetric case
\[
	p:=p_1=p_2=p_3,
	\quad \alpha:=\alpha_1=\alpha_2=\alpha_3\ge0,
	\quad \lambda:=\alpha_{12}=\alpha_{13}=\alpha_{23},
	\]
with \(0<\lambda<n\), and after writing \(s=1/(p-1)\), the system takes the form
\begin{equation}\label{eq:symm-system}
		\left\{
		\begin{aligned}
			U_1(x)&=C |x|^{-\alpha}\iint_{\Rn\times\Rn}
			\frac{U_2(y)^s |y|^{-\alpha}U_3(z)^s |z|^{-\alpha}}
			{|x-y|^\lambda |x-z|^\lambda |y-z|^\lambda}\,dy\,dz,\\
			U_2(y)&=C |y|^{-\alpha}\iint_{\Rn\times\Rn}
			\frac{U_1(x)^s |x|^{-\alpha}U_3(z)^s |z|^{-\alpha}}
			{|x-y|^\lambda |x-z|^\lambda |y-z|^\lambda}\,dx\,dz,\\
			U_3(z)&=C |z|^{-\alpha}\iint_{\Rn\times\Rn}
			\frac{U_1(x)^s |x|^{-\alpha}U_2(y)^s |y|^{-\alpha}}
			{|x-y|^\lambda |x-z|^\lambda |y-z|^\lambda}\,dx\,dy.
		\end{aligned}
		\right.
\end{equation}
In the fully symmetric setting, the scale condition \eqref{eq:scale} reduces to
$
1/p+(\alpha+\lambda)/n=1,
$
so that $p=n /(n-\alpha-\lambda)$ and $s=(n-\alpha-\lambda) / (\alpha+\lambda)$. Moreover, the pairwise admissibility condition is equivalent to $\alpha+\lambda \leq n / 2$. Under $\alpha \geq 0$ and $\lambda>0$, the remaining boundedness conditions are then automatic. Thus the extremizer theorem applies in the symmetric range $0<\alpha+\lambda \leq n / 2$, whereas the Kelvin-invariance statement below is an independent property of the scalar equation in the larger range $0<$ $\alpha+\lambda<n$.

\begin{theorem}\label{thm:diagonal}
	Assume that the hypotheses of Corollary~\ref{cor:trilinear-attainment} and
	\[
	p:=p_1=p_2=p_3,
	\quad \alpha:=\alpha_1=\alpha_2=\alpha_3\ge0,
	\quad \lambda:=\alpha_{12}=\alpha_{13}=\alpha_{23},
	\]
	where \(0<\lambda\leq(n / 2-\alpha)\) hold. Then every normalized nonnegative extremizing
	triple \((f_1,f_2,f_3)\) for \(C_A(p,p,p)\) satisfies
	\[
	f_1\equiv f_2\equiv f_3
	\quad\text{a.e. in }\Rn.
	\]
	Equivalently, for \(U_i:=f_i^{p-1}\), \(i=1,2,3\), as in
	\eqref{eq:symm-system}, one has \(U:=U_1\equiv U_2\equiv U_3\) a.e.
	With \(s=1/(p-1)\) and \(C=C_A(p,p,p)^{-1}\), the common function
	\(U\) satisfies
	\begin{equation}\label{eq:weighted-scalar}
		U(x)=C|x|^{-\alpha}\iint_{\Rn\times\Rn}
		\frac{U(y)^s|y|^{-\alpha}U(z)^s|z|^{-\alpha}}
		{|x-y|^\lambda |x-z|^\lambda |y-z|^\lambda}\,dy\,dz.
	\end{equation}
\end{theorem}
	
\begin{proposition}\label{prop:kelvin-weighted}
Assume \(0\le\alpha<n\), \(0<\lambda<n\), and \(\alpha+\lambda<n\), and set \(s=\frac{n-\alpha-\lambda}{\alpha+\lambda}\). Let $C>0$ be fixed, and let $V$ be a positive solution of
	\begin{equation}\label{eq:weighted-kelvin-eq}
		V(x)=C|x|^{-\alpha}\iint_{\Rn\times\Rn}
		\frac{V(y)^s|y|^{-\alpha}V(z)^s|z|^{-\alpha}}
		{|x-y|^\lambda |x-z|^\lambda |y-z|^\lambda}\,dy\,dz.
	\end{equation}
	For \(\rho>0\), define
	\(
	x^\rho:=\frac{\rho^2x}{|x|^2}, x\ne0,
	\)
	and
	\begin{equation}\label{eq:kelvin-transform}
		V_\rho(x):=\left(\frac\rho{|x|}\right)^{2(\alpha+\lambda)}V(x^\rho).
	\end{equation}
	Then \(V_\rho\), defined on \(\Rn\setminus\{0\}\) and arbitrarily at the origin, is also a solution in the a.e. sense.
\end{proposition} 

In the unweighted conformal case
\(
\alpha_i=0, \alpha_{ij}=n/3, p_i=3/2,
\)
Beckner's conformal multilinear inequality gives an explicit sharp constant and an explicit family of conformal extremizers.  We state this model case in Subsection~\ref{subsec:conformal-example}. 

We conclude the introduction with an outline of the proof of Theorem~\ref{thm:TA-attainment}. After symmetric decreasing rearrangement and a normalization by common dilations, the problem reduces to a maximizing sequence of nonnegative, radial, and radially nonincreasing functions. Passing to logarithmic radial variables and exploiting the scale condition yields an equivalent formulation involving a translation-invariant bilinear operator on \(\mathbb R\) whose kernel belongs to \(L^1(\mathbb R^2)\). This \(L^1\)-kernel structure, combined with a mixed weak--strong estimate, prevents the bilinear output from vanishing in the limit. A tightness argument based on the pairwise condition \eqref{eq:pair12} then prevents the two input functions from separating into different logarithmic scales. Finally, the Brezis--Lieb lemma and Helly's selection principle provide the required compactness and yield norm attainment for the bilinear operator.

The paper is organized as follows. 
Section~\ref{sec:preliminaries} records the rearrangement reduction, the logarithmic radial parametrization and isometries, and the Brezis–Lieb splitting lemma.
 Section~\ref{sec:compactness-attainment} proves norm attainment and trilinear attainment.  Section~\ref{sec:euler-diagonal} derives the Euler--Lagrange system, proves the diagonal reduction, establishes Kelvin invariance, and records the explicit unweighted conformal example.  Appendix~\ref{app:five-singularity} contains the five-singularity integral estimate used to prove \(L_0\in L^1(\R^2)\).

Throughout this paper, we use the following notation. For \(i\ne j\), we set
\(\alpha_{ji}=\alpha_{ij}\). For \(x\in\Rn\) and \(r>0\), let
\(B_r(x):=\{y\in\Rn:|y-x|<r\}\) and \(B_r=B_r(0)\). Let
\(\Snn:=\{x\in\Rn:|x|=1\}\) be equipped with surface measure \(d\sigma\), and
set \(\omega_{n-1}=\sigma(\Snn)\). We fix
\(e_1=(1,0,\ldots,0)\in\Snn\). For a measurable set \(E\), its indicator function is denoted by
\(\mathbf 1_E\). We write \(L_c^\infty(\Rn)\) for the space of
essentially bounded measurable functions on \(\Rn\) with compact support. For \(1\le p\le\infty\), we denote the \(L^p\)-norm by
\(\norm{\cdot}{p}\). \(L^{p,\infty}(\mathbb{R}^n)\) denotes the weak \(L^p\) space
equipped with the quasi-norm
\[
\norm{f}{L^{p,\infty}(\mathbb{R}^n)}
:=\sup_{\tau>0}\tau
\bigl|\{x\in \mathbb{R}^n:|f(x)|>\tau\}\bigr|^{1/p}.
\] \(C>0\) denotes a constant that may change from line to line. For \(A, B\ge 0\), the notation \(A\simeq B\) means that
\(cB\le A\le CB\) for constants \(c,C>0\). 

\section{Preliminaries}\label{sec:preliminaries}

We record the rearrangement reduction, fixes the logarithmic radial parametrization and isometries, and recalls the Brezis--Lieb splitting lemma. 
Throughout this section, \(f^*\) denotes the symmetric decreasing rearrangement of \(|f|\) on \(\Rn\).  The notation and the boundedness assumptions for \(\Lambda_A\) and \(T_A\) are those of Section~\ref{sec:introduction}.

We first recall the Brascamp--Lieb--Luttinger rearrangement inequality in the form needed for the trilinear kernel.
\begin{externaltheorem}{B}[Theorem 1.2, \cite{BrascampLiebLuttinger1974}]
\label{thm:BLL}
Let \(m,N\in\mathbb N\), let \(a_{\ell j}\in\R\), and let \(g_1,\ldots,g_N\) be nonnegative measurable functions on \(\Rn\).  For \(X=(x_1,\ldots,x_m)\in(\Rn)^m\), put
\[
L_\ell(X):=\sum_{j=1}^m a_{\ell j}x_j,
\quad \ell=1,\ldots,N.
\]
Then
\begin{equation}
\int_{(\Rn)^m}\prod_{\ell=1}^N g_\ell(L_\ell(X))\,dX
\le
\int_{(\Rn)^m}\prod_{\ell=1}^N g_\ell^*(L_\ell(X))\,dX,
\end{equation}
whenever the two sides are well defined. 
\end{externaltheorem}

The following consequence permits us to choose maximizing sequences for \(T_A\) with a common radial monotone structure.
\begin{lemma}
\label{lem:rearrangement}
Assume \eqref{eq:scale}--\eqref{eq:pisum} and \eqref{eq:nonneg}.  Then, for all \(f_1\in L^{p_1}(\Rn)\) and \(f_2\in L^{p_2}(\Rn)\),
\begin{equation}
\norm{T_A(f_1,f_2)}{p_3'}
\le
\norm{T_A(f_1^*,f_2^*)}{p_3'}.
\end{equation}
Consequently, the norm \(\norm{T_A}\) may be computed by restricting to nonnegative radial nonincreasing functions.
\end{lemma}

\begin{proof}
Since \(K_A\ge0\),
\[
|T_A(f_1,f_2)(x_3)|\le T_A(|f_1|,|f_2|)(x_3)
\quad\text{for a.e. }x_3\in\Rn.
\]
It is therefore enough to consider nonnegative \(f_1\) and \(f_2\).  For such functions, duality in \(L^{p_3'}\) gives
\[
\norm{T_A(f_1,f_2)}{p_3'}
=
\sup_{\substack{h\ge0\\ \norm{h}{p_3}=1}}
\Lambda_A(f_1,f_2,h).
\]
Thus it suffices to prove
\begin{equation}\label{eq:rearranged-trilinear-ineq}
\Lambda_A(f_1,f_2,h)
\le
\Lambda_A(f_1^*,f_2^*,h^*)
\end{equation}
for every nonnegative \(h\in L^{p_3}(\Rn)\).

Fix \(R>1\) and \(0<\varepsilon<1\).  For \(i=1,2,3\) and \(1\le i<j\le3\), set
\[
w_i^{R,\varepsilon}(x):=\min\{|x|^{-\alpha_i},\varepsilon^{-\alpha_i}\}\mathbf 1_{B_R}(x),
\]
and
\[
w_{ij}^{R,\varepsilon}(x):=\min\{|x|^{-\alpha_{ij}},\varepsilon^{-\alpha_{ij}}\}\mathbf 1_{B_R}(x).
\]
By \eqref{eq:nonneg}, these functions are nonnegative, radial, and radially nonincreasing.  Let \(\Lambda_A^{R,\varepsilon}\) be the form obtained from \(\Lambda_A\) by replacing the six singular factors with these truncated factors.  Applying Theorem~\ref{thm:BLL} to the nine factors
\[
f_1(x_1),\quad f_2(x_2),\quad h(x_3),
\]
\[
w_1^{R,\varepsilon}(x_1),\quad w_2^{R,\varepsilon}(x_2),\quad w_3^{R,\varepsilon}(x_3),
\]
and
\[
w_{12}^{R,\varepsilon}(x_1-x_2),\quad
w_{13}^{R,\varepsilon}(x_1-x_3),\quad
w_{23}^{R,\varepsilon}(x_2-x_3),
\]
whose linear arguments are
\[
x_1,
\quad x_2,
\quad x_3,
\quad x_1,
\quad x_2,
\quad x_3,
\quad x_1-x_2,
\quad x_1-x_3,
\quad x_2-x_3,
\]
yields
\[
\Lambda_A^{R,\varepsilon}(f_1,f_2,h)
\le
\Lambda_A^{R,\varepsilon}(f_1^*,f_2^*,h^*).
\]
Letting first \(R\to\infty\) and then \(\varepsilon \to 0^+\), the truncated factors increase pointwise to the original singular factors.  Monotone convergence gives \eqref{eq:rearranged-trilinear-ineq}.

Finally, by H\"older's inequality and preservation of \(L^{p_3}\)-norms under rearrangement,
\[
\Lambda_A(f_1^*,f_2^*,h^*)
\le
\norm{T_A(f_1^*,f_2^*)}{p_3'}\norm{h^*}{p_3}
=
\norm{T_A(f_1^*,f_2^*)}{p_3'}\norm{h}{p_3}.
\]
Taking the supremum over nonnegative \(h\) with \(\norm h{p_3}=1\) proves the claimed inequality.  The final assertion follows because rearrangement preserves the \(L^{p_i}\)-norms and does not decrease the operator norm quotient.
\end{proof}

In view of Lemma~\ref{lem:rearrangement}, throughout the remainder of the paper every maximizing sequence for \(T_A\) will, whenever needed, be chosen nonnegative, radial, and radially nonincreasing.  We also fix here the logarithmic radial notation used in Section~\ref{sec:compactness-attainment}.  Write
\(
x_i=e^{u_i}\xi_i,
\quad u_i\in\R,
\quad \xi_i\in\Snn.
\)
For a radial function \(f\) on \(\Rn\), define the logarithmic radial isometries
\begin{equation*}
\mathcal F_i f(u):=\omega_{n-1}^{1/p_i}e^{nu/p_i}f(e^u e_1),
\quad u\in\R,
\quad i=1,2.
\end{equation*}
For a radial function \(g\) on \(\Rn\), define
\begin{equation*}
\mathcal H g(u):=\omega_{n-1}^{1/p_3'}e^{nu/p_3'}g(e^u e_1).
\end{equation*}
Then
\begin{equation}\label{eq:F-isometry}
\norm{\mathcal F_i f}{p_i}=\norm{f}{p_i},
\quad i=1,2,
\end{equation}
and
\begin{equation}\label{eq:H-isometry}
\norm{\mathcal H g}{p_3'}=\norm{g}{p_3'}.
\end{equation}
Moreover, for radially nonincreasing \(f\),
\begin{equation}\label{eq:weak-identity}
\norm{\mathcal F_i f}{\infty}\simeq
\norm{f}{L^{p_i,\infty}(\Rn)},
\quad i=1,2,
\end{equation}
with constants depending only on \(n\) and \(p_i\).  We shall also use the standard strong-to-weak embedding, valid for \(h\in L^{p_i}(\Rn)\),
\begin{equation}\label{eq:Lp-weakLp}
\norm{h}{L^{p_i,\infty}(\Rn)}
\le \norm{h}{p_i},
\quad i=1,2.
\end{equation}

We shall also use the following form of the Brezis--Lieb lemma.
\begin{lemma}[\cite{BrezisLieb1983}]\label{lem:BL}
Let \(1<r<\infty\), and let \((u_j)\) be a bounded sequence in \(L^r(X)\) converging almost everywhere to \(u\).  Then
\[
\norm{u_j}{r}^r-\norm{u_j-u}{r}^r\to \norm{u}{r}^r.
\]
\end{lemma}

\section{Radial reduction, compactness, and norm attainment}\label{sec:compactness-attainment}

Using the logarithmic radial notation fixed in Section~\ref{sec:preliminaries}, we develop the reduced \(L^1\)-kernel framework and the compactness tools needed for Theorem~\ref{thm:TA-attainment}.  After proving the common-scale tightness lemma, we establish norm attainment and then obtain trilinear extremizers by duality.
We first establish the common dilation invariance of the trilinear form.

	\begin{lemma}\label{lem:dilation}
		Let $\rho>0$ and define
		$f_i^\rho(x):=\rho^{n/p_i}f_i(\rho x), i=1,2,3.
		$
		Assume \eqref{eq:scale}--\eqref{eq:pisum}, so that $T_A$ is a bounded bilinear operator by Theorem~\ref{thm:boundedness}. Then
		\[
		\Lambda_A(f_1^\rho,f_2^\rho,f_3^\rho)=\Lambda_A(f_1,f_2,f_3),
		\]
		and hence
		\[
		\norm{T_A(f_1^\rho,f_2^\rho)}{p_3'}=\norm{T_A(f_1,f_2)}{p_3'}.
		\]
	\end{lemma}
	
	\begin{proof}
		Insert the definition of $f_i^\rho$ into the trilinear form and make the change of variables $y_i=\rho x_i$. The total homogeneity exponent is
		\[
		\sum_{i=1}^3\frac n{p_i}-\left(3n-\sum_{i=1}^3\alpha_i-\alpha_{12}-\alpha_{13}-\alpha_{23}\right),
		\]
		which vanishes precisely by \eqref{eq:scale}. The corresponding operator statement is obtained by duality.
	\end{proof}

	\subsection{The isometric reduced kernel}

	\begin{lemma}
		\label{lem:five}
		Let $n\ge1$, let $e_1\in\Rn$ be a fixed unit vector, and assume
		\[
		0\le \mu_1,\mu_2,\nu_1,\nu_2,\lambda<n.
		\]
		Suppose, in addition, that
		\begin{align}
			\mu_1+\mu_2+\lambda&<2n,\label{eq:five1}\\
			\nu_1+\nu_2+\lambda&<2n,\label{eq:five2}\\
			\mu_1+\nu_1+\lambda&>n,\label{eq:five3}\\
			\mu_2+\nu_2+\lambda&>n,\label{eq:five4}\\
			\mu_1+\mu_2+\nu_1+\nu_2+\lambda&>2n.\label{eq:five5}
		\end{align}
		Then
		\begin{equation}\label{eq:five-int}
			\iint_{\Rn\times\Rn}\frac{\,dx\,dy}{|x|^{\mu_1}|y|^{\mu_2}|x-e_1|^{\nu_1}|y-e_1|^{\nu_2}|x-y|^\lambda}<\infty.
		\end{equation}
	\end{lemma}
	
	\begin{proof}
		See Appendix~\ref{app:five-singularity}.
	\end{proof}
	
	\begin{lemma}
		\label{lem:L0}
		Assume \eqref{eq:scale}--\eqref{eq:pisum} and \eqref{eq:nonneg}. Then there exists a nonnegative kernel
		\[
		L_0\in L^1(\R^2)
		\]
		such that, for bounded compactly supported radial functions $f_1,f_2$, one has
		\begin{equation}\label{eq:reduced-identity}
			\mathcal H(T_A(f_1,f_2))(u)=B_0(\mathcal F_1f_1,\mathcal F_2f_2)(u)
		\end{equation}
		for a.e. $u\in\R$, where
		\begin{equation}\label{eq:B0}
			B_0(F,G)(u):=\iint_{\R^2}L_0(s,t)F(u+s)G(u+t)\,dsdt.
		\end{equation}
		More precisely, if
		\[
		\beta_1:=\frac n{p_1'}-\alpha_1-\frac{\alpha_{12}+\alpha_{13}}2,
		\quad
		\beta_2:=\frac n{p_2'}-\alpha_2-\frac{\alpha_{12}+\alpha_{23}}2,
		\]
		then $L_0$ is represented for a.e. $(s,t)\in\R^2$ by
		\begin{equation}\label{eq:L0formula}
		\begin{aligned}
			L_0(s,t)&=\omega_{n-1}^{\frac1{p_3'}-\frac1{p_1}-\frac1{p_2}}
			e^{\beta_1s+\beta_2t} \\
			&\times\iint_{(\Snn)^2}
			\frac{d\sigma(\xi_1)d\sigma(\xi_2)}
			{|e^{s/2}\xi_1-e^{-s/2}e_1|^{\alpha_{13}}
				|e^{t/2}\xi_2-e^{-t/2}e_1|^{\alpha_{23}}
				|e^{(s-t)/2}\xi_1-e^{(t-s)/2}\xi_2|^{\alpha_{12}}}.
		\end{aligned}
		\end{equation}
		The spherical integral in \eqref{eq:L0formula} may diverge on the null set \(\{s=0\}\cup\{t=0\}\cup\{s=t\}\); on this set we define \(L_0\) arbitrarily. This choice does not affect the \(L^1\) representative nor the operator \(B_0\).
	\end{lemma}
	
	\begin{proof}
		We first derive the exact reduced representation. For radial $f_1,f_2$, write
		\[
		x_1=e^{u_1}\xi_1,
		\quad x_2=e^{u_2}\xi_2,
		\quad x_3=e^{u_3}e_1.
		\]
		Using
		\[
		dx_i=e^{nu_i}\,du_i\,d\sigma(\xi_i),\quad
		f_i(e^{u_i}\xi_i)=\omega_{n-1}^{-1/p_i}e^{-nu_i/p_i}\mathcal F_i f_i(u_i),
		\]
		we obtain
		\[
		\mathcal H(T_A(f_1,f_2))(u_3)
		=\iint_{\R^2}\mathcal F_1f_1(u_1)\mathcal F_2f_2(u_2)\widetilde L_0(u_1,u_2,u_3)\,du_1du_2,
		\]
		where
		\begin{align*}
			\widetilde L_0(u_1,u_2,u_3)
			&=\omega_{n-1}^{\frac1{p_3'}-\frac1{p_1}-\frac1{p_2}}
			e^{\left(\frac n{p_1'}-\alpha_1-\frac{\alpha_{12}+\alpha_{13}}2\right)u_1}
			e^{\left(\frac n{p_2'}-\alpha_2-\frac{\alpha_{12}+\alpha_{23}}2\right)u_2} \\
			&\quad\times
			e^{\left(\frac n{p_3'}-\alpha_3-\frac{\alpha_{13}+\alpha_{23}}2\right)u_3}
			Z(u_1-u_3,u_2-u_3),
		\end{align*}
		and
		\[
		Z(s,t):=\iint_{(\Snn)^2}
		\frac{d\sigma(\xi_1)d\sigma(\xi_2)}
		{|e^{s/2}\xi_1-e^{-s/2}e_1|^{\alpha_{13}}
			|e^{t/2}\xi_2-e^{-t/2}e_1|^{\alpha_{23}}
			|e^{(s-t)/2}\xi_1-e^{(t-s)/2}\xi_2|^{\alpha_{12}}}.
		\]
		By the scaling balance \eqref{eq:scale},
		\[
		\frac n{p_3'}-\alpha_3-\frac{\alpha_{13}+\alpha_{23}}2
		=-\left(\frac n{p_1'}-\alpha_1-\frac{\alpha_{12}+\alpha_{13}}2\right)
		-\left(\frac n{p_2'}-\alpha_2-\frac{\alpha_{12}+\alpha_{23}}2\right).
		\]
		Therefore $\widetilde L_0(u_1,u_2,u_3)=L_0(u_1-u_3,u_2-u_3)$ with $L_0$ given by \eqref{eq:L0formula}. This proves the reduced identity for bounded compactly supported radial functions.
		
		It remains to prove $L_0\in L^1(\R^2)$. Since $L_0\ge0$, Tonelli's theorem allows us to compute its $L^1$ norm directly. Put
		\[
		x=e^s\xi_1,
		\quad y=e^t\xi_2.
		\]
		Then
		\[
		ds\,d\sigma(\xi_1)=|x|^{-n}\,dx,
		\quad dt\,d\sigma(\xi_2)=|y|^{-n}\,dy.
		\]
		Moreover,
		\[
		|e^{s/2}\xi_1-e^{-s/2}e_1|=e^{-s/2}|x-e_1|,
		\quad
		|e^{t/2}\xi_2-e^{-t/2}e_1|=e^{-t/2}|y-e_1|,
		\]
		and
		\[
		|e^{(s-t)/2}\xi_1-e^{(t-s)/2}\xi_2|=e^{-(s+t)/2}|x-y|.
		\]
		Hence, with $c_\omega=\omega_{n-1}^{\frac1{p_3'}-\frac1{p_1}-\frac1{p_2}}$,
		\begin{equation}\label{eq:L0-L1-transform}
			\int_{\R^2}L_0(s,t)\,dsdt
			=c_\omega\iint_{\Rn\times\Rn}\frac{\,dx\,dy}{|x|^{\mu_1}|y|^{\mu_2}|x-e_1|^{\nu_1}|y-e_1|^{\nu_2}|x-y|^\lambda},
		\end{equation}
		where
		\[
		\mu_1:=\frac n{p_1}+\alpha_1,
		\quad \mu_2:=\frac n{p_2}+\alpha_2,
		\quad \nu_1:=\alpha_{13},
		\quad \nu_2:=\alpha_{23},
		\quad \lambda:=\alpha_{12}.
		\]
		We verify the hypotheses of Lemma~\ref{lem:five}. First, $0\le\nu_1,\nu_2,\lambda<n$ follows from \eqref{eq:nonneg} and \eqref{eq:pairlt}, while $0\le\mu_i<n$ follows from $\alpha_i\ge0$ and $\alpha_i<n/p_i'$ in \eqref{eq:onelt}. Moreover, $\mu_1+\mu_2+\lambda<2n$ is equivalent, by \eqref{eq:scale}, to
		\[
		\alpha_3+\alpha_{13}+\alpha_{23}>\frac n{p_3'},
		\]
		which is \eqref{eq:tailgt} with $i=3$. Next, $\nu_1+\nu_2+\lambda<2n$ is precisely \eqref{eq:pairsum}. Also, $\mu_1+\nu_1+\lambda>n$ is equivalent to \eqref{eq:tailgt} with $i=1$, and $\mu_2+\nu_2+\lambda>n$ is equivalent to \eqref{eq:tailgt} with $i=2$. Finally, again by \eqref{eq:scale},
		\[
		\mu_1+\mu_2+\nu_1+\nu_2+\lambda=2n+\left(\frac n{p_3'}-\alpha_3\right),
		\]
		which is strictly larger than $2n$ by \eqref{eq:onelt} with $i=3$. 
		Therefore Lemma~\ref{lem:five} applies, and the right-hand side of
		\eqref{eq:L0-L1-transform} is finite. Hence the extended nonnegative
		measurable representative defined by \eqref{eq:L0formula} belongs to
		\(L^1(\R^2)\) and is finite for a.e. \((s,t)\in\R^2\). We fix this finite
		\(L^1\) representative, after assigning arbitrary finite values on the
		exceptional null set. Thus \eqref{eq:L0formula} is an a.e. formula, and
		the reduced identity \eqref{eq:reduced-identity} is understood in the
		usual a.e. sense.
	\end{proof}
	
	\begin{remark}\label{rem:L0-extension-and-pairwise}
		The identity in Lemma~\ref{lem:L0} is stated only for bounded compactly
		supported radial functions. Its extension to arbitrary radial
		\(L^{p_i}\)-functions is obtained in Lemma~\ref{lem:norm-equiv} by density
		and boundedness. We also note that the pairwise admissibility condition
		\eqref{eq:pair12} is not used in the proof of \(L_0\in L^1(\R^2)\);
		it enters only later in the compactness argument.
	\end{remark}
	
	\begin{lemma}\label{lem:norm-equiv}
		Assume that $1<p_i<\infty$ for $i=1,2,3$, together with
		\eqref{eq:scale}--\eqref{eq:pisum} and \eqref{eq:nonneg} hold. Let $L_0$
		be the reduced kernel from Lemma~\ref{lem:L0}, and let $B_0$ be
		initially defined by \eqref{eq:B0} on bounded compactly supported
		functions. Then $B_0$ extends uniquely to a bounded bilinear operator
		\[
		B_0:L^{p_1}(\R)\times L^{p_2}(\R)\to L^{p_3'}(\R),
		\]
		and
		\begin{equation}\label{eq:norm-equiv}
			\norm{B_0}_{L^{p_1}(\R)\times L^{p_2}(\R)\to L^{p_3'}(\R)}
			=
			\norm{T_A}_{L^{p_1}(\Rn)\times L^{p_2}(\Rn)\to L^{p_3'}(\Rn)}.
		\end{equation}
	\end{lemma}
	
	\begin{proof}
Let $\mathcal S_i$ denote the space of bounded compactly supported functions on
$\R$, viewed as a dense subspace of $L^{p_i}(\R)$, $i=1,2$. For
$F\in\mathcal S_1$ and $G\in\mathcal S_2$, define the corresponding radial lifts
\[
f_1(x)=\omega_{n-1}^{-1/p_1}|x|^{-n/p_1}F(\log|x|),
\quad
f_2(x)=\omega_{n-1}^{-1/p_2}|x|^{-n/p_2}G(\log|x|),
\]
for $x\ne0$. Since $F$ and $G$ are bounded and compactly supported in the
logarithmic variable, the functions $f_1$ and $f_2$ are bounded compactly
supported radial functions on $\Rn$. Moreover,
\[
\mathcal F_1 f_1=F,\quad \mathcal F_2 f_2=G,
\quad
\norm{f_i}{p_i}=\norm{\mathcal F_i f_i}{p_i}.
\]
By Lemma~\ref{lem:L0},
\[
\mathcal H(T_A(f_1,f_2))=B_0(F,G)
\]
in $L^{p_3'}(\R)$. Hence
\[
\norm{B_0(F,G)}{p_3'}
=
\norm{T_A(f_1,f_2)}{p_3'}
\le
\norm{T_A}\norm{F}{p_1}\norm{G}{p_2}.
\]
Thus $B_0$ is bounded on $\mathcal S_1\times\mathcal S_2$ with operator norm
at most $\norm{T_A}$. Since $\mathcal S_i$ is dense in $L^{p_i}(\R)$, $B_0$
extends uniquely to a bounded bilinear operator
\[
B_0:L^{p_1}(\R)\times L^{p_2}(\R)\to L^{p_3'}(\R),
\]
and
\[
\norm{B_0}\le \norm{T_A}.
\]

We next record the corresponding extension of the reduced identity. Let
$f_i\in L^{p_i}(\Rn)$ be radial and set $F_i=\mathcal F_i f_i$. Choose
$F_{i,k}\in\mathcal S_i$ with $F_{i,k}\to F_i$ in $L^{p_i}(\R)$, and let
$f_{i,k}$ be the radial lifts of $F_{i,k}$. Then
$f_{i,k}\to f_i$ in $L^{p_i}(\Rn)$. The boundedness of $T_A$, together with
the boundedness of the extended operator $B_0$, allows us to pass to the limit
in
\[
\mathcal H(T_A(f_{1,k},f_{2,k}))=B_0(F_{1,k},F_{2,k}).
\]
Consequently,
\begin{equation}\label{eq:extended-reduced-identity}
\mathcal H(T_A(f_1,f_2))=B_0(\mathcal F_1 f_1,\mathcal F_2 f_2)
\end{equation}
in $L^{p_3'}(\R)$ for all radial $f_i\in L^{p_i}(\Rn)$.

Now let $f_i\in L^{p_i}(\Rn)$ be arbitrary. By Lemma~\ref{lem:rearrangement},
\[
\norm{T_A(f_1,f_2)}{p_3'}
\le
\norm{T_A(f_1^*,f_2^*)}{p_3'}.
\]
Applying the extended reduced identity to the radial functions $f_i^*$ gives
\[
\norm{T_A(f_1^*,f_2^*)}{p_3'}
=
\norm{B_0(\mathcal F_1 f_1^*,\mathcal F_2 f_2^*)}{p_3'}.
\]
Therefore
\[
\norm{T_A(f_1,f_2)}{p_3'}
\le
\norm{B_0}\norm{\mathcal F_1 f_1^*}{p_1}
\norm{\mathcal F_2 f_2^*}{p_2}
=
\norm{B_0}\norm{f_1}{p_1}
\norm{f_2}{p_2}.
\]
Taking the supremum over nonzero $f_1$ and $f_2$ yields
\[
\norm{T_A}\le \norm{B_0}.
\]
Together with the already proved reverse inequality, this proves
\eqref{eq:norm-equiv}.
\end{proof}

Choosing nonzero nonnegative test functions with small pairwise separated
supports away from the origin gives \(T_A\not\equiv0\). Hence
\begin{equation}\label{eq:B0-positive}
\norm{B_0}=\norm{T_A}>0.
\end{equation}

	\subsection{A mixed weak--strong estimate}
	
	\begin{lemma}\label{lem:mixed-kernel}
		Let $1<p_1,p_2<\infty$, $1<r<\infty$, and assume that there exist
		$q_1,q_2$ such that
		\begin{equation}\label{eq:qchoice}
			p_i\le q_i<\infty\quad (i=1,2),
			\quad \frac1{q_1}+\frac1{q_2}=1.
		\end{equation}
		For $L\in L^1(\R^2)$, define
		\[
		B_L(U,V)(u):=\iint_{\R^2}L(s,t)U(u+s)V(u+t)\,dsdt.
		\]
		Then, for every $U\in L^{p_1}(\R)\cap L^\infty(\R)$ and
		$V\in L^{p_2}(\R)\cap L^\infty(\R)$,
		\begin{equation}\label{eq:weakstrong-reduced}
			\norm{B_L(U,V)}{r}
			\le \norm{L}{1}\prod_{i=1}^2
			\norm{U_i}{p_i}^{\theta_i}
			\norm{U_i}{\infty}^{1-\theta_i},
		\end{equation}
		where $U_1=U$, $U_2=V$, and
		\[
		\theta_i:=\frac{p_i}{q_i r}\in(0,1),\quad i=1,2.
		\]
		If, in addition, $1<p_3<\infty$, \eqref{eq:scale}--\eqref{eq:pisum},
		\eqref{eq:nonneg}, and \eqref{eq:pair12} hold, let $L=L_0$ be the
		reduced kernel from Lemma~\ref{lem:L0} and set $r=p_3'$. Then there
		exists $C>0$ such that
		\begin{equation}\label{eq:mixed-physical}
			\norm{T_A(f_1,f_2)}{p_3'}
			\le C\prod_{i=1}^2
			\norm{f_i}{p_i}^{\theta_i}
			\norm{f_i}{L^{p_i,\infty}(\Rn)}^{1-\theta_i}
		\end{equation}
		for all nonnegative radially nonincreasing functions
		$f_i\in L^{p_i}(\Rn)$, $i=1,2$.
	\end{lemma}
	
	\begin{proof}
		By Tonelli's theorem and H\"older's inequality with exponents
		$q_1,q_2$,
		\begin{align*}
			\norm{B_L(U,V)}{1}
			&\le \int_{\R}\iint_{\R^2}|L(s,t)|
			|U(u+s)| |V(u+t)|\,ds\,dt\,du \\
			&\le \norm{L}{1}
			\norm{U}{q_1}\norm{V}{q_2}.
		\end{align*}
		Also,
		\[
		\norm{B_L(U,V)}{\infty}
		\le \norm{L}{1}\norm{U}{\infty}
		\norm{V}{\infty}.
		\]
		Interpolation between $L^1$ and $L^\infty$ therefore gives
		\[
		\norm{B_L(U,V)}{r}
		\le \norm{L}{1}
		\norm{U}{q_1}^{1/r}
		\norm{V}{q_2}^{1/r}
		\norm{U}{\infty}^{1-1/r}
		\norm{V}{\infty}^{1-1/r}.
		\]
		Since $p_i\le q_i<\infty$,
		\[
		\norm{U_i}{q_i}
		\le \norm{U_i}{p_i}^{p_i/q_i}
		\norm{U_i}{\infty}^{1-p_i/q_i},
		\quad i=1,2,
		\]
		which proves \eqref{eq:weakstrong-reduced}. The bounds on $q_i$ and
		$r$ imply $0<\theta_i<1$.
		
		Under the additional hypotheses, \eqref{eq:pair12} guarantees that
		exponents $q_1,q_2$ satisfying \eqref{eq:qchoice} can be chosen. For
		nonnegative radial nonincreasing $f_i$, the extended reduced identity
		\eqref{eq:extended-reduced-identity} and the isometries
		\eqref{eq:F-isometry}--\eqref{eq:H-isometry} give
		\[
		\norm{T_A(f_1,f_2)}{p_3'}
		=
		\norm{B_{L_0}(\mathcal F_1f_1,\mathcal F_2f_2)}{p_3'}.
		\]
		Applying \eqref{eq:weakstrong-reduced} with $L=L_0$ and then using
		the radial weak-type identity \eqref{eq:weak-identity} proves
		\eqref{eq:mixed-physical}.
	\end{proof}
	
	We call \((u_j,v_j)\) a normalized maximizing sequence for a bounded
	bilinear operator \(S:X_1\times X_2\to Y\) if
	\(\|u_j\|_{X_1}=\|v_j\|_{X_2}=1\) and
	\(\|S(u_j,v_j)\|_Y\to\|S\|\).
	
	\begin{remark}\label{rem:weak-vanish}
		If $\norm{f_i}{p_i}=1$ for $i=1,2$, then \eqref{eq:mixed-physical} gives
		\[
		\norm{T_A(f_1,f_2)}{p_3'}\le C
		\norm{f_1}{L^{p_1,\infty}(\mathbb{R}^n)}^{1-\theta_1}
		\norm{f_2}{L^{p_2,\infty}(\mathbb{R}^n)}^{1-\theta_2}.
		\]
		Consequently, for a normalized maximizing sequence, the product of the two weak norms cannot tend to zero. By \eqref{eq:Lp-weakLp}, each weak norm is uniformly bounded above; hence neither weak norm can vanish along a maximizing sequence.
	\end{remark}
	
	\subsection{Common scale nonvanishing}

	\begin{lemma}\label{lem:tightness}
Assume that \eqref{eq:scale}--\eqref{eq:pisum}, \eqref{eq:nonneg},
		and \eqref{eq:pair12} hold. 
		Let \(\{(f_{1,j},f_{2,j})\}_{j\ge1}\) be a normalized maximizing
		sequence for \(T_A\).
		Assume in addition that each \(f_{i,j}\) is nonnegative, radial, and radially nonincreasing, and set
		\[
		\Phi_{i,j}:=\mathcal F_i f_{i,j},\quad i=1,2.
		\]
		Then there exist \(R>0\), \(\eta>0\), \(j_0\in \mathbb{N}\) and a sequence \(\tau_j\in\R\) such that for \(j\ge j_0\)
		\begin{equation}\label{eq:tight}
			\int_{-R}^R|\Phi_{1,j}(u+\tau_j)|^{p_1}\,du\ge\eta,
			\quad
			\int_{-R}^R|\Phi_{2,j}(u+\tau_j)|^{p_2}\,du\ge\eta.
		\end{equation}
	\end{lemma}
	
\begin{proof}
By \eqref{eq:F-isometry}, we obtain 
\(\norm{\Phi_{i,j}}{p_i}=1\) for \(i=1,2\).
Moreover, \eqref{eq:extended-reduced-identity} and
\eqref{eq:H-isometry} yield
\[
\norm{B_0(\Phi_{1,j},\Phi_{2,j})}{p_3'}
=
\norm{T_A(f_{1,j},f_{2,j})}{p_3'}
\rightarrow
\norm{T_A}
=
\norm{B_0},
\]
where the last identity follows from \eqref{eq:norm-equiv}. Hence
\(\{(\Phi_{1,j},\Phi_{2,j})\}_{j\ge1}\) is a normalized maximizing
sequence for \(B_0\). By \eqref{eq:B0-positive}, there exists
\(J\in\mathbb N\) such that
\begin{equation}\label{eq:B0-lower}
\norm{B_0(\Phi_{1,j},\Phi_{2,j})}{p_3'}
\ge \frac12 \norm{B_0}
\end{equation}
for any \(j\ge J\).

Since each \(f_{i,j}\) is nonnegative, radial and radially nonincreasing, \eqref{eq:weak-identity}
and \eqref{eq:Lp-weakLp} imply that there exists $C>0$ such that
\begin{equation}\label{eq:Phi-Linf}
\norm{\Phi_{i,j}}{\infty}\le C,
\quad i=1,2.
\end{equation}
Since \(L_0\ge 0\) and \(L_0 \in L^1(\mathbb R^2)\), for \(\varepsilon>0\), there exists \(K\in L_c^\infty(\mathbb R^2)\) such that \(K\ge0\) and
\begin{equation}\label{eq:Kapprox}
\norm{L_0-K}{1}\le \varepsilon .
\end{equation}

Define
\[
B_K(F,G)(u)
:=
\iint_{\mathbb R^2}K(s,t)F(u+s)G(u+t)\,ds\,dt .
\]
By \eqref{eq:pair12}, \(p_2\le p_1'\), and therefore \eqref{eq:qchoice}
holds for \(p_1\) and \(p_2\). 
By Lemma~\ref{lem:mixed-kernel}, together with
\(\norm{\Phi_{i,j}}{p_i}=1\), \eqref{eq:Phi-Linf}, and
\eqref{eq:Kapprox}, we obtain 
\[
\norm{
B_0(\Phi_{1,j},\Phi_{2,j})
-
B_K(\Phi_{1,j},\Phi_{2,j})
}{p_3'}
\le C\varepsilon,
\]
where \(C\) is independent of \(j\) and \(\varepsilon\). Choosing
\(\varepsilon>0\) sufficiently small, we deduce
from \eqref{eq:B0-lower} that
\begin{equation}\label{eq:BK-lower}
\norm{B_K(\Phi_{1,j},\Phi_{2,j})}{p_3'}
\ge \norm{B_0}/4>0
\end{equation}
for all \(j\ge J\).

Since \(K\in L_c^\infty(\mathbb R^2)\), there exists \(R>0\) such that \(\operatorname{supp}K\subset[-R,R]^2\). Define
\[A_{i,j}(u):=\int_{u-R}^{u+R}|\Phi_{i,j}(v)|^{p_i}\,dv \quad i=1,2.\]
We claim that
\begin{equation}\label{eq:overlap-claim}
\liminf_{j\to\infty}
\sup_{u\in\mathbb R}\min\{A_{1,j}(u),A_{2,j}(u)\}>0.
\end{equation}
Suppose by contradiction that, after passing to a subsequence, 
\begin{equation}\label{eq:no-both-fixed}
\sup_{u\in\mathbb R}\min\{A_{1,j}(u),A_{2,j}(u)\}\rightarrow0.
\end{equation}
Since \(\operatorname{supp}K\subset[-R,R]^2\), we have
\[
B_K(\Phi_{1,j},\Phi_{2,j})(u)
=
\int_{-R}^R\int_{-R}^R
K(s,t)\Phi_{1,j}(u+s)\Phi_{2,j}(u+t)\,ds\,dt .
\]
Thus, by H\"older's inequality,
\begin{align*}
|B_K(\Phi_{1,j},\Phi_{2,j})(u)|
&\le \norm{K}{\infty}
\left(\int_{-R}^R|\Phi_{1,j}(u+s)|\,ds\right)
\left(\int_{-R}^R|\Phi_{2,j}(u+t)|\,dt\right) \\
&\le
\norm{K}{\infty}
(2R)^{1/p_1'+1/p_2'}
A_{1,j}(u)^{1/p_1}A_{2,j}(u)^{1/p_2} \\
&\le
C A_{1,j}(u)^{1/p_1}A_{2,j}(u)^{1/p_2}.
\end{align*}

Writing \(a:=p_3'/p_1\) and \(b:=p_3'/p_2\),
we obtain
\begin{equation}\label{eq:BK-local}
\norm{B_K(\Phi_{1,j},\Phi_{2,j})}{p_3'}^{p_3'}
\le
C\int_{\mathbb R}A_{1,j}(u)^aA_{2,j}(u)^b\,du .
\end{equation}
Moreover, by \eqref{eq:pair12},
\(a+b=p_3'(1/p_1+1/p_2)\ge p_3'>1\).
Choose \(0<\delta<\min\{a,b,a+b-1\}\).
From \eqref{eq:Phi-Linf}, there exists a constant \(M_R\), independent of
\(j\) and \(u\), such that \(0\le A_{i,j}(u)\le M_R\), \(i=1,2\).
Moreover, by Fubini's theorem and \(\norm{\Phi_{i,j}}{p_i}=1\), we obtain 
\begin{equation}\label{eq:Aij-mass}
\int_{\mathbb R} A_{i,j}(u)\,du
=
\int_{\mathbb R}\int_{u-R}^{u+R}|\Phi_{i,j}(v)|^{p_i}\,dv\,du
=
2R,
\quad i=1,2.
\end{equation}

We claim that, for \(0\le x,y\le M_R\) and \(0<\zeta\le1\), if
\(\min\{x,y\}\le\zeta\), then
\begin{equation}\label{eq:small-product}
x^a y^b\le C\zeta^\delta(x+y),
\end{equation}
where \(C\) depends only on \(a,b,\delta\), and \(M_R\). Indeed, if
\(x\le\zeta\), then
\[
x^a y^b
\le
\zeta^\delta x^{a-\delta}y^b
\le
\zeta^\delta (x+y)^{a+b-\delta}
\le
C\zeta^\delta(x+y),
\]
because \(a+b-\delta>1\) and \(x+y\le 2M_R\). The case \(y\le\zeta\) is
analogous.

Fix \(\zeta\in(0,1]\). By \eqref{eq:no-both-fixed}, for \(j\) sufficiently
large, \(\min\{A_{1,j}(u),A_{2,j}(u)\}\le \zeta\) for any
\(u\in\mathbb R\).
Using \eqref{eq:small-product} and \eqref{eq:Aij-mass}, we obtain 
\begin{equation}
\int_{\mathbb R}A_{1,j}(u)^aA_{2,j}(u)^b\,du
\le
C\zeta^\delta
\int_{\mathbb R}\bigl(A_{1,j}(u)+A_{2,j}(u)\bigr)\,du 
\le
C\zeta^\delta.
\end{equation}
Combining this with \eqref{eq:BK-local}, we obtain
\(\norm{B_K(\Phi_{1,j},\Phi_{2,j})}{p_3'}^{p_3'}\le C\zeta^\delta\)
for sufficiently large \(j\). Choosing \(\zeta>0\) so small that
\(\norm{B_K(\Phi_{1,j},\Phi_{2,j})}{p_3'}<\norm{B_0}/4\) contradicts \eqref{eq:BK-lower}. Hence
\eqref{eq:overlap-claim} holds.

It follows from \eqref{eq:overlap-claim} that there exist \(\eta>0\) and 
\(\tau_j\in\mathbb R\) such that \(A_{1,j}(\tau_j)\ge\eta\) and
\(A_{2,j}(\tau_j)\ge\eta\) for sufficiently large \(j\).
Since
\[
A_{i,j}(\tau_j)
=
\int_{\tau_j-R}^{\tau_j+R}|\Phi_{i,j}(v)|^{p_i}\,dv
=
\int_{-R}^{R}|\Phi_{i,j}(u+\tau_j)|^{p_i}\,du,
\] we obtain 
\eqref{eq:tight}.
\end{proof}
	
	\subsection{Proof of Theorem~\ref{thm:TA-attainment} and Corollary~\ref{cor:trilinear-attainment}}\label{subsec:proofs-attainment}
	
		\begin{proof}[Proof of Theorem~\ref{thm:TA-attainment}]
		Let \(\{(f_{1,j},f_{2,j})\}_{j\ge1}\) be a normalized maximizing
		sequence for $T_A$. By Lemma~\ref{lem:rearrangement}, we may assume
		that each $f_{i,j}$ is nonnegative, radial, and radially nonincreasing. For
		\(\Phi_{i,j}:=\mathcal F_i f_{i,j}\), by Lemma~\ref{lem:tightness}, there exist \(R,\eta>0\) and \(\tau_j\in\R\) such that
		\[
		\int_{-R}^R|\Phi_{i,j}(u+\tau_j)|^{p_i}\,du\ge\eta,
		\quad i=1,2,
		\]
		for sufficiently large \(j\). By the definition of \(f_i^\rho\)
		and Lemma~\ref{lem:dilation},
		\(\{(f_{1,j}^{e^{\tau_j}},f_{2,j}^{e^{\tau_j}})\}_{j\ge1}\) is also
		a normalized maximizing sequence for \(T_A\). Since
		\[
		\mathcal F_i(f_{i,j}^{e^{\tau_j}})(u)=\Phi_{i,j}(u+\tau_j),
		\quad i=1,2,
		\]
		we may therefore assume
		\begin{equation}\label{eq:proof-tight}
			\int_{-R}^R|\Phi_{1,j}(u)|^{p_1}\,du\ge\eta,
			\quad
			\int_{-R}^R|\Phi_{2,j}(u)|^{p_2}\,du\ge\eta
		\end{equation}
		for all sufficiently large $j$. Moreover, \((\Phi_{1,j},\Phi_{2,j})\) is a normalized maximizing sequence for \(B_0\).
		
		Write \(f_{i,j}(x)=\varphi_{i,j}(|x|)\). Since \(f_{i,j}\) is
		nonnegative, radial, and radially nonincreasing, for \(i=1,2\) and
\(s>0\), 
\[
1=\norm{f_{i,j}}{p_i}^{p_i}
\ge \int_{\{|x|\le s\}} f_{i,j}(x)^{p_i}\,dx 
\ge \frac{\omega_{n-1}}{n}s^n
\varphi_{i,j}(s)^{p_i}.
\]
This implies that
		\begin{equation}\label{eq:radial-decay-maximizing-sequence}
		0\le \varphi_{i,j}(s)\le
		\left(\frac{n}{\omega_{n-1}}\right)^{1/p_i}s^{-n/p_i}.
		\end{equation}
		For each integer \(m\ge2\) and \(i=1,2\), \eqref{eq:radial-decay-maximizing-sequence} implies that \(\{\varphi_{i,j}\}\) 
is uniformly bounded on \([m^{-1},m]\). Combining this with the fact that each $\varphi_{i,j}$ is nonincreasing, it follows from Helly's selection theorem
		\cite[Sec.~36.5, Theorem~5]{KolmogorovFomin1975}, applied successively to
\(\{\varphi_{i,j}\}\), \(i=1,2\), on \([m^{-1},m]\), and a diagonal extraction in \(m\) that there exist nonnegative nonincreasing functions
		\(\varphi_i\) such that, up to a subsequence,
		\[
		\varphi_{i,j}(\rho)\to\varphi_i(\rho)
		\quad\text{for a.e. }\rho>0,
		\quad i=1,2 .
		\]
		Set \(f_i(x)=\varphi_i(|x|)\) for \(x\ne0\), with arbitrary value at
		the origin. Thus, \(f_{i,j}\to f_i\) a.e. on \(\Rn\), and hence
	\begin{equation}\label{eq:Phi-ae-convergence}
		\Phi_{i,j}(u)\to
		\Phi_i(u)=\omega_{n-1}^{1/p_i}e^{nu/p_i}f_i(e^u e_1)
		\quad\text{for a.e. }u\in\R .
	\end{equation}
		Fatou's lemma yields \(\Phi_i\in L^{p_i}(\R)\) and
		\(\norm{\Phi_i}{p_i}\le1\). By \eqref{eq:Phi-Linf}, \eqref{eq:proof-tight} and \eqref{eq:Phi-ae-convergence}, the dominated convergence theorem yields \(\Phi_1\not\equiv0\) and
		\(\Phi_2\not\equiv0\).
		
		Set \(M=\norm{B_0}=\norm{T_A}>0\), \(r=p_3'\), and
		\(a_i=\norm{\Phi_i}{p_i}\). We have \(0<a_i\le1\). By \eqref{eq:Phi-Linf}, \eqref{eq:Phi-ae-convergence}, and \(L_0\in L^1(\mathbb R^2)\), the dominated convergence theorem yields
		\(B_0(\Phi_{1,j},\Phi_{2,j})(u)\to B_0(\Phi_1,\Phi_2)(u)\).
		Moreover, \eqref{eq:Phi-Linf} and \eqref{eq:Phi-ae-convergence} imply
		\(\Phi_i\in L^\infty(\R)\). Set \(h_{i,j}=\Phi_{i,j}-\Phi_i\). Then
		\(h_{i,j}\to0\) a.e. and
		\begin{equation}\label{eq:hij-bounds}
		\{h_{i,j}\} \text{ is bounded in } L^{p_i}(\R)\cap L^\infty(\R),
		\quad i=1,2.
		\end{equation}
		By the Brezis--Lieb lemma,
		\begin{equation}\label{eq:hij-BL}
		\norm{h_{i,j}}{p_i}^{p_i}=1-a_i^{p_i}+o(1),
		\quad i=1,2.
		\end{equation}
		Applying Brezis--Lieb again to 
		\(B_0(\Phi_{1,j},\Phi_{2,j})\), we obtain
		\begin{equation}\label{eq:B0-BL-splitting}
		\norm{B_0(\Phi_{1,j},\Phi_{2,j})}{r}^r
		=
		\norm{B_0(\Phi_1,\Phi_2)}{r}^r
		+\norm{B_0(\Phi_{1,j},\Phi_{2,j})-B_0(\Phi_1,\Phi_2)}{r}^r
		+o(1).
		\end{equation}
		Noting that
		\[
		B_0(\Phi_{1,j},\Phi_{2,j})-B_0(\Phi_1,\Phi_2)
		=
		B_0(h_{1,j},h_{2,j})+B_0(\Phi_1,h_{2,j})+B_0(h_{1,j},\Phi_2),
		\]
		we obtain
		\begin{equation}\label{eq:B0-cross-term-estimate}
		\begin{aligned}
		&\norm{B_0(\Phi_{1,j},\Phi_{2,j})-B_0(\Phi_1,\Phi_2)
		-B_0(h_{1,j},h_{2,j})}{r} \\
		&\le
		\norm{B_0(\Phi_1,h_{2,j})}{r}
		+\norm{B_0(h_{1,j},\Phi_2)}{r}.
		\end{aligned}
		\end{equation}
		We claim that
		\[
		\norm{B_0(\Phi_1,h_{2,j})}{r}
		+\norm{B_0(h_{1,j},\Phi_2)}{r}\to0.
		\]
	 By \eqref{eq:pair12}, fix \(q_i\) with \(p_i\le q_i<\infty\),
\(i=1,2\), and \(1/q_1+1/q_2=1\). Fix
		\(\varepsilon>0\), choose \(\Psi\in L_c^\infty(\R)\) and
		\(K\in L_c^\infty(\R^2)\) satisfying
		\[
		\norm{\Psi}{\infty}\le 2\norm{\Phi_1}{\infty},
		\quad
		\norm{\Phi_1-\Psi}{p_1}+\norm{L_0-K}{1}<\varepsilon.
		\]
	By \eqref{eq:weakstrong-reduced} and \eqref{eq:hij-bounds}, we obtain
		\[
		\begin{aligned}
		&\norm{B_0(\Phi_1,h_{2,j})-B_K(\Psi,h_{2,j})}{r} \\
		\le &
		\norm{B_0(\Phi_1-\Psi,h_{2,j})}{r}
		+\norm{B_{L_0-K}(\Psi,h_{2,j})}{r} \\
		\le &
		\norm{B_0}\norm{\Phi_1-\Psi}{p_1}\norm{h_{2,j}}{p_2}
		+C\norm{L_0-K}{1}
		\le C\varepsilon,
		\end{aligned}
		\]
		where \(C>0\) is independent of \(j\) and \(\varepsilon\).
		Fubini's theorem and the dominated convergence theorem yield \(B_K(\Psi,h_{2,j})\to0\) a.e.
		Since \(K\) and \(\Psi\) are compactly supported, it follows from
		\eqref{eq:hij-bounds} that there exist a compact interval \(I\subset\R\)
		and \(C>0\), both independent of \(j\), such that
		\[
		\operatorname{supp} B_K(\Psi,h_{2,j})\subset I,
		\quad
		\norm{B_K(\Psi,h_{2,j})}{\infty}\le C.
		\]
		Thus the dominated convergence theorem yields
\(
\norm{B_K(\Psi,h_{2,j})}{r}\to0.
		\)
		Consequently,
		\begin{equation}\label{eq:Phi1-h2-limsup}
		\limsup_{j\to\infty}\norm{B_0(\Phi_1,h_{2,j})}{r}\le C\varepsilon.
		\end{equation}
		Letting \(\varepsilon\to0^+\), \eqref{eq:Phi1-h2-limsup} yields
		\(
		\norm{B_0(\Phi_1,h_{2,j})}{r}\to0.
		\)
		The same argument yields
		\(		\norm{B_0(h_{1,j},\Phi_2)}{r}\to0.
\)  This completes the claim. Combining this with \eqref{eq:B0-cross-term-estimate}, we obtain
		\begin{equation}\label{eq:B0-difference-reduction}
\norm{B_0(\Phi_{1,j},\Phi_{2,j})-B_0(\Phi_1,\Phi_2)}{r}
=
\norm{B_0(h_{1,j},h_{2,j})}{r}+o(1).
		\end{equation}

		By  \eqref{eq:hij-BL},
		\begin{equation}\label{eq:h1h2-limsup}
		\limsup_{j\to\infty}\norm{B_0(h_{1,j},h_{2,j})}{r}
		\le
		M(1-a_1^{p_1})^{1/p_1}(1-a_2^{p_2})^{1/p_2}.
		\end{equation}
		Combining \eqref{eq:B0-difference-reduction}, \eqref{eq:h1h2-limsup},
		and \eqref{eq:B0-BL-splitting} with
		\(\norm{B_0(\Phi_{1,j},\Phi_{2,j})}{r}\to M\), we obtain
		\begin{equation}\label{eq:limit-splitting-M}
		M^r
		\le
		\norm{B_0(\Phi_1,\Phi_2)}{r}^r
		+M^r(1-a_1^{p_1})^{r/p_1}(1-a_2^{p_2})^{r/p_2}.
		\end{equation}
		Set \(m_1=a_1^{p_1}\), \(m_2=a_2^{p_2}\), \(\gamma_1=r/p_1\), and
		\(\gamma_2=r/p_2\). By \eqref{eq:pair12} and \(r=p_3'>1\), one has
		\(\gamma_1+\gamma_2>1\). Since
		\(\norm{B_0(\Phi_1,\Phi_2)}{r}\le Ma_1a_2\), \eqref{eq:limit-splitting-M}
		yields
\begin{equation}\label{eq:xy-lower}
		1\le m_1^{\gamma_1} m_2^{\gamma_2}+(1-m_1)^{\gamma_1}(1-m_2)^{\gamma_2}.
		\end{equation}
		The weighted AM--GM inequality
		\cite[Theorem~9]{HardyLittlewoodPolya1952} yields
		\[
		m_1^{\gamma_1} m_2^{\gamma_2}\le \sigma^{\gamma_1+\gamma_2},\quad
		(1-m_1)^{\gamma_1}(1-m_2)^{\gamma_2}\le(1-\sigma)^{\gamma_1+\gamma_2},
		\]
	where \(\sigma=\frac{\gamma_1 m_1+\gamma_2 m_2}{\gamma_1+\gamma_2}\).
		Noting that \(0\le\sigma\le1\) and \(\gamma_1+\gamma_2>1\), we obtain 
		\[
		m_1^{\gamma_1} m_2^{\gamma_2}+(1-m_1)^{\gamma_1}(1-m_2)^{\gamma_2}
		\le \sigma^{\gamma_1+\gamma_2}+(1-\sigma)^{\gamma_1+\gamma_2}\le1.
		\]
		In view of \eqref{eq:xy-lower}, $\sigma^{\gamma_1+\gamma_2}+(1-\sigma)^{\gamma_1+\gamma_2}=1$. Since
		\(\gamma_1+\gamma_2>1\), 
		\(\sigma\in\{0,1\}\). The positivity of \(a_1,a_2\) excludes
		\(\sigma=0\), so \(\sigma=1\). Since \(m_1,m_2\le1\), \(m_1=m_2=1\), which implies \(a_1=a_2=1\). By
		\eqref{eq:limit-splitting-M}, we deduce that 
		\(\norm{B_0(\Phi_1,\Phi_2)}{r}=M\).
		For \(f_i\), it follows from \eqref{eq:F-isometry}, \eqref{eq:norm-equiv}, and \eqref{eq:extended-reduced-identity} that
		\[
		\norm{f_i}{p_i}=1,
		\quad
		\norm{T_A(f_1,f_2)}{p_3'}
		=\norm{B_0(\Phi_1,\Phi_2)}{p_3'}=
		\norm{T_A}.
		\]
		This completes the proof.
	\end{proof}
	
	\begin{proof}[Proof of Corollary~\ref{cor:trilinear-attainment}]
		Without loss of generality, we may assume \eqref{eq:pair12}.
		By Theorem~\ref{thm:TA-attainment}, there exist nonnegative functions
		\(f_i\in L^{p_i}(\Rn)\), \(i=1,2\), such that
		\[
		\norm{f_1}{p_1}=\norm{f_2}{p_2}=1,\quad
		\norm{T_A(f_1,f_2)}{p_3'}=\norm{T_A}.
		\]
		Set \(g=T_A(f_1,f_2)\in L^{p_3'}(\Rn)\). Then \(g\ge0\) a.e. and \(\norm{g}{p_3'}=\norm{T_A}>0\). Define
		\(f_3=g^{p_3'-1}/\norm{g}{p_3'}^{p_3'-1}\).
		A direct computation yields
		\(\norm{f_3}{p_3}=1\) and
		\[
		\Lambda_A(f_1,f_2,f_3)
		=\int_{\Rn}g(x)f_3(x)\,dx
		=\frac{\int_{\Rn}g(x)^{p_3'}\,dx}{\norm{g}{p_3'}^{p_3'-1}}
		=\norm{g}{p_3'}
		=\norm{T_A}
		=C_A(p_1,p_2,p_3).
		\]
		Thus \(C_A(p_1,p_2,p_3)\) is attained by the nonnegative normalized
		triple \((f_1,f_2,f_3)\).
		
		By \eqref{eq:rearranged-trilinear-ineq}, we obtain 
		\[
		C_A(p_1,p_2,p_3)
		\ge
		\Lambda_A(f_1^*,f_2^*,f_3^*)
		\ge
		\Lambda_A(f_1,f_2,f_3)
		=C_A(p_1,p_2,p_3).
		\]
		Thus \(C_A(p_1,p_2,p_3)\) is attained by
		\((f_1^*,f_2^*,f_3^*)\). 
		This proves the corollary.
	\end{proof}
	
\section{Diagonal reduction and Kelvin invariance}\label{sec:euler-diagonal}

This section derives the Euler--Lagrange system, proves the diagonal reduction, establishes Kelvin invariance, and records the explicit unweighted conformal example. 
We first derive the Euler--Lagrange system for normalized extremizing triples.

\begin{proposition}\label{prop:EL}
	Let \(f_i\in L^{p_i}(\Rn)\), \(i=1,2,3\), be nonnegative functions such that
	\[
	\norm{f_i}{p_i}=1,\quad i=1,2,3,\quad
	\Lambda_A(f_1,f_2,f_3)=C_A(p_1,p_2,p_3).
	\]
	Then \((f_1,f_2,f_3)\) satisfies the following system almost everywhere:
	\begin{align}
		C_A(p_1,p_2,p_3)f_1(x_1)^{p_1-1}
		&=\iint_{\Rn\times\Rn}K_A(x_1,x_2,x_3)f_2(x_2)f_3(x_3)\,dx_2\,dx_3,\label{eq:EL1}\\
		C_A(p_1,p_2,p_3)f_2(x_2)^{p_2-1}
		&=\iint_{\Rn\times\Rn}K_A(x_1,x_2,x_3)f_1(x_1)f_3(x_3)\,dx_1\,dx_3,\label{eq:EL2}\\
		C_A(p_1,p_2,p_3)f_3(x_3)^{p_3-1}
		&=\iint_{\Rn\times\Rn}K_A(x_1,x_2,x_3)f_1(x_1)f_2(x_2)\,dx_1\,dx_2.\label{eq:EL3}
	\end{align}
\end{proposition}

\begin{proof}
	Write \(C_A=C_A(p_1,p_2,p_3)\). We prove only the first identity, since the other two are obtained by the same argument.
	
	Define
	\[
	\ell_1(\varphi):=\Lambda_A(\varphi,f_2,f_3),
	\quad \varphi\in L^{p_1}(\Rn).
	\]
	The sharp inequality and the normalizations of \(f_2\) and \(f_3\) give
	\[
	|\ell_1(\varphi)|\le C_A\norm{\varphi}{p_1}.
	\]
	Hence \(\ell_1\in (L^{p_1})^*=L^{p_1'}\). Let \(G_1\in L^{p_1'}(\Rn)\) be its representing function:
	\[
	\ell_1(\varphi)=\int_{\Rn}G_1(x_1)\varphi(x_1)\,dx_1.
	\]
	By the standard truncation of the nonnegative kernel, followed by Tonelli's theorem and monotone convergence, this representative may be chosen so that
	\[
	G_1(x_1)=\iint_{\Rn\times\Rn}
	K_A(x_1,x_2,x_3)f_2(x_2)f_3(x_3)\,dx_2\,dx_3
	\]
	for almost every \(x_1\).
	
	Let \(\varphi\in L^{p_1}(\Rn)\cap L_c^\infty(\Rn)\) be real-valued and set
	\[
		F_\varepsilon:=\frac{f_1+\varepsilon\varphi}{\norm{f_1+\varepsilon\varphi}{p_1}}.
		\]
	For \(|\varepsilon|\) small, \(\norm{F_\varepsilon}{p_1}=1\), and therefore
	\[
	\Phi(\varepsilon):=\Lambda_A(F_\varepsilon,f_2,f_3)
	\]
	has a local maximum at \(\varepsilon=0\). By linearity in the first component,
	\[
		\Phi(\varepsilon)
		=\frac{C_A+\varepsilon\displaystyle\int_{\Rn}G_1(x_1)\varphi(x_1)\,dx_1}{\norm{f_1+\varepsilon\varphi}{p_1}}.
		\]
	Since \(p_1>1\) and \(\norm{f_1}{p_1}=1\), the \(L^{p_1}\)-norm is differentiable at \(f_1\) and
	\[
	\left.\frac{d}{d \varepsilon}\right|_{\varepsilon=0}\left\|f_1+\varepsilon \varphi\right\|_{p_1}=\int_{\Rn}f_1(x_1)^{p_1-1}\varphi(x_1)\,dx_1.
	\]
	The condition \(\Phi'(0)=0\) gives
	\[
	\int_{\Rn}\Bigl(G_1(x_1)-C_A f_1(x_1)^{p_1-1}\Bigr)\varphi(x_1)\,dx_1=0.
	\]
	By the density of \(L^{p_1}(\Rn)\cap L_c^\infty(\Rn)\) in \(L^{p_1}(\Rn)\), and since
	\(G_1-C_Af_1^{p_1-1}\in L^{p_1'}(\Rn)\), it follows that
	\(G_1=C_Af_1^{p_1-1}\) almost everywhere. This proves \eqref{eq:EL1}.
	
	Repeating the same constrained variation in the second and third variables gives \eqref{eq:EL2} and \eqref{eq:EL3}.
\end{proof}

Set
\[
U_i(x):=f_i(x)^{p_i-1},
\qquad s_i:=\frac{1}{p_i-1},
\qquad i=1,2,3.
\]
Then \(f_i=U_i^{s_i}\).
Consequently, \eqref{eq:EL1}--\eqref{eq:EL3} can be rewritten, with the common constant \(C:=C_A(p_1,p_2,p_3)^{-1}\), as
\begin{equation}\label{eq:general-system}
	\left\{
	\begin{aligned}
		U_1(x_1)&=C\iint_{\Rn\times\Rn}K_A(x_1,x_2,x_3)U_2(x_2)^{s_2}U_3(x_3)^{s_3}\,dx_2dx_3,\\
		U_2(x_2)&=C\iint_{\Rn\times\Rn}K_A(x_1,x_2,x_3)U_1(x_1)^{s_1}U_3(x_3)^{s_3}\,dx_1dx_3,\\
		U_3(x_3)&=C\iint_{\Rn\times\Rn}K_A(x_1,x_2,x_3)U_1(x_1)^{s_1}U_2(x_2)^{s_2}\,dx_1dx_2.
	\end{aligned}
	\right.
\end{equation}

\subsection{The fully symmetric weighted system and two lemmas}
We next specialize \eqref{eq:general-system} to the fully symmetric weighted case
\begin{equation}\label{eq:fully-symm}
		\alpha:=\alpha_1=\alpha_2=\alpha_3,
		\quad \lambda:=\alpha_{12}=\alpha_{13}=\alpha_{23},
		\quad p:=p_1=p_2=p_3.
	\end{equation}
	Then \(s:=s_1=s_2=s_3=1/(p-1)\). 
Upon identifying \(x_1\) with \(x\), \(x_2\) with \(y\), and \(x_3\) with \(z\), \eqref{eq:general-system} becomes
\begin{equation}
	\left\{
	\begin{aligned}
		U_1(x)&=C |x|^{-\alpha}\iint_{\Rn\times\Rn}
		\frac{U_2(y)^s |y|^{-\alpha}U_3(z)^s |z|^{-\alpha}}
		{|x-y|^\lambda |x-z|^\lambda |y-z|^\lambda}\,dy\,dz,\\
		U_2(y)&=C |y|^{-\alpha}\iint_{\Rn\times\Rn}
		\frac{U_1(x)^s |x|^{-\alpha}U_3(z)^s |z|^{-\alpha}}
		{|x-y|^\lambda |x-z|^\lambda |y-z|^\lambda}\,dx\,dz,\\
		U_3(z)&=C |z|^{-\alpha}\iint_{\Rn\times\Rn}
		\frac{U_1(x)^s |x|^{-\alpha}U_2(y)^s |y|^{-\alpha}}
		{|x-y|^\lambda |x-z|^\lambda |y-z|^\lambda}\,dx\,dy.
	\end{aligned}
	\right.
\end{equation}

\begin{lemma}\label{lem:Riesz-positive}
	Let $0<\lambda<n$. For every real-valued $\Psi \in C_c^{\infty}\left(\mathbb{R}^n\right)$,
	\begin{equation}\label{eq:Riesz-positive}
		\iint_{\Rn\times\Rn}\frac{\Psi(x)\Psi(y)}{|x-y|^\lambda}\,dx\,dy
		=c_{n,\lambda}\int_{\Rn}|\widehat\Psi(\xi)|^2|\xi|^{\lambda-n}\,d\xi\ge0,
	\end{equation}
	where \(c_{n,\lambda}>0\) depends only on \(n,\lambda\) and on the Fourier transform convention.
\end{lemma}

\begin{proof}
	Since \(|x|^{-\lambda}\in L^1_{\rm loc}(\mathbb R^n)\) and
	\(\Psi\in C_c^\infty(\mathbb R^n)\), the double integral is finite and
	\[
	\iint_{\mathbb R^n\times\mathbb R^n}
	|x-y|^{-\lambda}\Psi(x)\Psi(y)\,dx\,dy
	=\langle |x|^{-\lambda}*\Psi,\Psi\rangle .
	\]
	By the standard Riesz-kernel formula,
	\[
	\widehat{|\cdot|^{-\lambda}}(\xi)=c_{n,\lambda}|\xi|^{\lambda-n}
	\quad\text{in } \mathcal S'(\mathbb R^n),
	\quad c_{n,\lambda}>0 .
	\]
	Using the Parseval identity in the \(\mathcal S'\)-\(\mathcal S\) pairing gives
	\[
	\langle |x|^{-\lambda}*\Psi,\Psi\rangle
	=C_{\mathcal F}c_{n,\lambda}
	\int_{\mathbb R^n}
	|\xi|^{\lambda-n}\widehat\Psi(\xi)\widehat\Psi(-\xi)\,d\xi .
	\]
	Since \(\Psi\) is real-valued, \(\widehat\Psi(-\xi)=\overline{\widehat\Psi(\xi)}\). Absorbing the positive
	Fourier-convention constant \(C_{\mathcal F}\) into \(c_{n,\lambda}\), we obtain
	\[
	\iint_{\mathbb R^n\times\mathbb R^n}
	\frac{\Psi(x)\Psi(y)}{|x-y|^\lambda}\,dx\,dy
	=
	{c_{n,\lambda}}\int_{\mathbb R^n}
	|\widehat\Psi(\xi)|^2|\xi|^{\lambda-n}\,d\xi\ge 0 .
	\]
\end{proof}

\begin{lemma}\label{lem:riesz-abs-positive}
	If \(0<\lambda<n\), \(\Psi\) is a real-valued measurable function satisfying
	\[
	I_\lambda(\Psi):=\iint_{\Rn\times\Rn}
	\frac{|\Psi(x)|\,|\Psi(y)|}{|x-y|^\lambda}\,dx\,dy<\infty,
	\]
	then the signed Riesz form is absolutely convergent and
	\begin{equation}\label{eq:absolute-riesz-positive}
		\iint_{\Rn\times\Rn}
		\frac{\Psi(x)\Psi(y)}{|x-y|^\lambda}\,dx\,dy\ge0.
	\end{equation}
\end{lemma}

\begin{proof}
	It suffices first to extend Lemma~\ref{lem:Riesz-positive} from smooth functions to bounded compactly supported functions. Let \(\psi\in L^\infty_c(\Rn)\) be real-valued. 
	Define
	\(\psi_\varepsilon:=\rho_\varepsilon*\psi,
	\delta_\varepsilon:=\psi_\varepsilon-\psi .
	\)
	For sufficiently small \(\varepsilon\), there exists a compact set
	\(K\subset\Rn\) and constants \(A,D<\infty\), independent of
	\(\varepsilon\), such that
	\[
	\operatorname{supp}\psi,\operatorname{supp}\psi_\varepsilon,
	\operatorname{supp}\delta_\varepsilon\subset K,
	\]
	and
	\[
	\|\psi\|_\infty+\|\psi_\varepsilon\|_\infty\leq A,
	\quad
	\|\delta_\varepsilon\|_\infty\leq D,
	\quad
	\|\delta_\varepsilon\|_{1}\to0
	\quad\text{as }\varepsilon\to0 .
	\]
	
	For \(H\in\{\psi,\psi_\varepsilon\}\) and every \(\eta>0\), we have
		\[
		\begin{aligned}
			\iint_{\Rn\times\Rn}
			\frac{|\delta_\varepsilon(x)|\,|H(y)|}{|x-y|^\lambda}\,dx\,dy
			&=
			\iint_{|x-y|\le\eta}
			\frac{|\delta_\varepsilon(x)|\,|H(y)|}{|x-y|^\lambda}\,dx\,dy\\
			&\quad+
			\iint_{|x-y|>\eta}
			\frac{|\delta_\varepsilon(x)|\,|H(y)|}{|x-y|^\lambda}\,dx\,dy\\
			&\le
			DA|K|\int_{B_\eta}|z|^{-\lambda}\,dz
			+\eta^{-\lambda}A|K|\|\delta_\varepsilon\|_{1}\\
			&=
			\frac{DA|K|\,|\Snn|}{n-\lambda}\eta^{n-\lambda}
			+\eta^{-\lambda}A|K|\|\delta_\varepsilon\|_{1}.
		\end{aligned}
		\]
		For fixed \(\eta>0\), letting \(\varepsilon\to0\) gives
		\[
		\limsup_{\varepsilon\to0}
		\iint_{\Rn\times\Rn}
		\frac{|\delta_\varepsilon(x)|\,|H(y)|}{|x-y|^\lambda}\,dx\,dy
		\le
		\frac{DA|K|\,|\Snn|}{n-\lambda}\eta^{n-\lambda},
		\quad H\in\{\psi,\psi_\varepsilon\}.
		\]
		Since \(\eta>0\) is arbitrary, letting \(\eta\to0\) yields
		\[
		\iint_{\Rn\times\Rn}
		\frac{|\delta_\varepsilon(x)|\,|H(y)|}{|x-y|^\lambda}\,dx\,dy
		\to0
		\quad(\varepsilon\to0),
		\quad H\in\{\psi,\psi_\varepsilon\}.
		\]
		Therefore
		\[
		\begin{aligned}
			&\left|
			\iint_{\Rn\times\Rn}
			\frac{\psi_\varepsilon(x)\psi_\varepsilon(y)-\psi(x)\psi(y)}
			{|x-y|^\lambda}\,dx\,dy
			\right|\\
			&\quad\le
			\iint_{\Rn\times\Rn}
			\frac{|\delta_\varepsilon(x)|\,|\psi_\varepsilon(y)|}
			{|x-y|^\lambda}\,dx\,dy
			+
			\iint_{\Rn\times\Rn}
			\frac{|\delta_\varepsilon(x)|\,|\psi(y)|}
			{|x-y|^\lambda}\,dx\,dy\\
			&\quad\to0
			\quad(\varepsilon\to0).
		\end{aligned}
		\]
	Since the quadratic form is nonnegative on \(\psi_\varepsilon\) by Lemma~\ref{lem:Riesz-positive}, it is also nonnegative on \(\psi\).
	
	Now let \(\Psi\) satisfy \(I_\lambda(\Psi)<\infty\). For \(M,R>1\), put
	\[
	\Psi_{M,R}(x):=\max\{-M,\min\{\Psi(x),M\}\}\mathbf 1_{B_R}(x).
	\]
	Since \(\Psi_{M,R}\in L_c^\infty(\Rn)\),
	\[
	\iint_{\Rn\times\Rn}
	\frac{\Psi_{M,R}(x)\Psi_{M,R}(y)}{|x-y|^\lambda}\,dx\,dy\ge0.
	\]
	Moreover, \(\Psi_{M,R}\to\Psi\) a.e. as \(M,R\to\infty\), and
	\[
	\left|
	\frac{\Psi_{M,R}(x)\Psi_{M,R}(y)-\Psi(x)\Psi(y)}{|x-y|^\lambda}
	\right|
	\le
	2\frac{|\Psi(x)|\,|\Psi(y)|}{|x-y|^\lambda}.
	\]
	Since \(I_\lambda(\Psi)<\infty\), the dominated convergence theorem gives
	\[
	\iint_{\Rn\times\Rn}
	\frac{\Psi_{M,R}(x)\Psi_{M,R}(y)}{|x-y|^\lambda}\,dx\,dy
	\to
	\iint_{\Rn\times\Rn}
	\frac{\Psi(x)\Psi(y)}{|x-y|^\lambda}\,dx\,dy
	\quad\text{as }M,R\to\infty.
	\]
	Passing to the limit proves \eqref{eq:absolute-riesz-positive}.
\end{proof}

\subsection{Proof of reduction results}

	\begin{proof}[Proof of Theorem~\ref{thm:diagonal}]
		
		Set \(s=1/(p-1)\) and \(U_i=f_i^{p-1}\), \(i=1,2,3\). By
		Proposition~\ref{prop:EL}, \((U_1,U_2,U_3)\) satisfies
		\eqref{eq:symm-system} with \(C=C_A(p,p,p)^{-1}\). Since the
		extremizing triple is normalized and the kernel is strictly
		positive, \(U_1,U_2,U_3\) are positive and finite a.e.
		
		We first establish the integrability needed below. Since
		\(U_i^s=f_i\), the elementary estimate
		\[
		|X^{p-1}-Y^{p-1}|\,|X-Y|
		\le C_p(X^p+Y^p),
		\quad X,Y\ge0,\quad p>1,
		\]
		gives
		\[
		|(U_1-U_2)(U_1^s-U_2^s)|
		=
		|f_1^{p-1}-f_2^{p-1}|\,|f_1-f_2|
		\le C_p(f_1^p+f_2^p)\in L^1(\Rn).
		\]
		Moreover, the trilinear Stein--Weiss boundedness theorem, applied to
		\(|f_1-f_2|,|f_1-f_2|,f_3\), gives
		\begin{equation}\label{integ est}
		\begin{aligned}
			&\iiint_{(\Rn)^3}
			\frac{|U_1(x)^s-U_2(x)^s|\,|U_1(y)^s-U_2(y)^s|U_3(z)^s}
			{|x|^\alpha |y|^\alpha |z|^\alpha
				|x-y|^\lambda |x-z|^\lambda |y-z|^\lambda}
			\,dx\,dy\,dz\\
			&\quad\le
			C_A(p,p,p)\norm{f_1-f_2}{p}^2
			\norm{f_3}{p}<\infty,
		\end{aligned}
		\end{equation}
		Replacing \((f_1,f_2,f_3)\) by \((f_2,f_3,f_1)\) gives the
		corresponding estimates for \(U_2,U_3,U_1\).
		
		We now compare the first two equations in
		\eqref{eq:symm-system}. 
		After relabeling the free variable in the second equation as $x$, subtract the second equation from the first,
		we obtain
		\[
		U_1(x)-U_2(x)
		=
		-C|x|^{-\alpha}\iint_{\Rn\times\Rn}
		\frac{\bigl(U_1(y)^s-U_2(y)^s\bigr)|y|^{-\alpha}
			U_3(z)^s|z|^{-\alpha}}
		{|x-y|^\lambda |x-z|^\lambda |y-z|^\lambda}\,dy\,dz.
		\]
		For \(M,R>1\), define
		\[
		D_{M,R}(x):=
		\max\{-M,\min\{U_1(x)^s-U_2(x)^s,M\}\}
		\mathbf 1_{B_R}(x).
		\]
		Then
		\[
		|D_{M,R}|\le |U_1^s-U_2^s|,
		\quad
		D_{M,R}\to U_1^s-U_2^s
		\quad\text{a.e. as }M,R\to\infty.
		\]
		The preceding two integrability estimates permit us to multiply
		the difference equation by \(D_{M,R}(x)\), integrate in \(x\),
		and apply Fubini's theorem. Thus
		\begin{equation}\label{eq:test-trunc}
			\int_{\Rn}(U_1-U_2)D_{M,R}\,dx = -C\iiint_{(\Rn)^3} \frac{D_{M,R}(x)(U_1(y)^s-U_2(y)^s)U_3(z)^s \, dx\,dy\,dz}{|x|^\alpha |y|^\alpha |z|^\alpha |x-y|^\lambda |x-z|^\lambda |y-z|^\lambda}.
		\end{equation}
		Letting \(M,R\to\infty\) and using the dominated convergence
		theorem gives
		\begin{equation}\label{eq:test-limit}
			\begin{aligned}
				\int_{\Rn}(U_1-U_2)(U_1^s-U_2^s)\,dx
				= &-C\iiint_{(\Rn)^3}
				\frac{\bigl(U_1(x)^s-U_2(x)^s\bigr)\bigl(U_1(y)^s-U_2(y)^s\bigr)}
				{|x|^\alpha |y|^\alpha |z|^\alpha |x-y|^\lambda} \\
				&\times
				\frac{U_3(z)^s}{|x-z|^\lambda |y-z|^\lambda}\,dx\,dy\,dz.
			\end{aligned}
		\end{equation}
		
		We now explain why the triple integral in
		\eqref{eq:test-limit} is nonnegative. By the absolute-integrability
		estimate \ref{integ est} and Tonelli's theorem,
		\begin{align*}
			&\int_{\Rn}U_3(z)^s|z|^{-\alpha}
			I_\lambda\!\left(
			\bigl(U_1(x)^s-U_2(x)^s\bigr)
			|x|^{-\alpha}|x-z|^{-\lambda}
			\right)\,dz\\
			&\quad=
			\iiint_{(\Rn)^3}
			\frac{|U_1(x)^s-U_2(x)^s|\,
				|U_1(y)^s-U_2(y)^s|U_3(z)^s}
			{|x|^\alpha |y|^\alpha |z|^\alpha
				|x-y|^\lambda |x-z|^\lambda |y-z|^\lambda}
			\,dx\,dy\,dz<\infty.
		\end{align*}
		Since \(U_3(z)^s|z|^{-\alpha}>0\) for a.e. \(z\), it follows that
		\[
		I_\lambda\!\left(
		\bigl(U_1(x)^s-U_2(x)^s\bigr)
		|x|^{-\alpha}|x-z|^{-\lambda}
		\right)<\infty
		\quad\text{for a.e. }z\in\Rn.
		\]
		For each such \(z\), the function inside
		\(I_\lambda\) is real-valued and measurable. Hence
		Lemma~\ref{lem:riesz-abs-positive} applies and gives
		\[
		\iint_{\Rn\times\Rn}
		\frac{
			\bigl(U_1(x)^s-U_2(x)^s\bigr)
			\bigl(U_1(y)^s-U_2(y)^s\bigr)}
		{|x|^\alpha |y|^\alpha
			|x-z|^\lambda |y-z|^\lambda |x-y|^\lambda}
		\,dx\,dy\ge0
		\]
		for a.e. \(z\). Finally, the same absolute-integrability estimate
		permits Fubini's theorem for the signed integral, so
		\begin{align*}
			&\iiint_{(\Rn)^3}
			\frac{\bigl(U_1(x)^s-U_2(x)^s\bigr)
				\bigl(U_1(y)^s-U_2(y)^s\bigr)U_3(z)^s}
			{|x|^\alpha |y|^\alpha |z|^\alpha |x-y|^\lambda |x-z|^\lambda |y-z|^\lambda}\,dx\,dy\,dz\\
			&\quad=
			\int_{\Rn}U_3(z)^s|z|^{-\alpha}
			\Bigg[
			\iint_{\Rn\times\Rn}
			\frac{\bigl(U_1(x)^s-U_2(x)^s\bigr)
				\bigl(U_1(y)^s-U_2(y)^s\bigr)\,dx\,dy}
			{|x|^\alpha |y|^\alpha |x-z|^\lambda |y-z|^\lambda |x-y|^\lambda}
			\Bigg]dz\ge0.
		\end{align*}
		
		On the other hand, strict monotonicity of \(t\mapsto t^s\) on
		\((0,\infty)\) gives
		\[
		(U_1-U_2)(U_1^s-U_2^s)\ge0\quad\text{a.e.}
		\]
		Hence the left-hand side of \eqref{eq:test-limit} is nonnegative,
		whereas its right-hand side is nonpositive. Equality follows, and
		thus
		\[
		(U_1-U_2)(U_1^s-U_2^s)=0
		\quad\text{a.e. in }\Rn.
		\]
		Consequently, \(U_1=U_2\) a.e. Repeating the same argument with
		the second and third equations gives \(U_2=U_3\) a.e. Therefore
		\(U_1=U_2=U_3\) a.e. Since \(f_i=U_i^s\), we obtain
		\(f_1=f_2=f_3\) a.e. Substituting \(U_1=U_2=U_3=:U\) into
		\eqref{eq:symm-system} gives \eqref{eq:weighted-scalar}.
	\end{proof}

The diagonal reduction above identifies the scalar equation \eqref{eq:weighted-scalar}.
The present argument uses only this reduction and does not require a classification of its solutions.

\subsection{Weighted and unweighted Kelvin invariance}\label{subsec:kelvin-invariance}

\begin{proof}[Proof of Proposition~\ref{prop:kelvin-weighted}]
	
	Set $d:=\alpha+\lambda$ and $I(x):=\rho^2x/|x|^2$. Then $I$ is an involution on $\Rn\setminus\{0\}$, with
	\[
	|I(a)-I(b)|=\frac{\rho^2|a-b|}{|a||b|},
	\quad dI(y)=\left(\frac\rho{|y|}\right)^{2n}dy.
	\]
	Fix $x\ne0$ and put $X=I(x)$. In the right-hand side of \eqref{eq:weighted-kelvin-eq} for $V_\rho$, set $Y=I(y)$ and $Z=I(z)$. Then
	\[
	|y|=\frac{\rho^2}{|Y|},
	\quad |z|=\frac{\rho^2}{|Z|},
	\quad dy=\left(\frac\rho{|Y|}\right)^{2n}dY,
	\quad dz=\left(\frac\rho{|Z|}\right)^{2n}dZ,
	\]
	and
	\[
	|x-y|=\frac{|x|}{|Y|}|X-Y|,
	\quad |x-z|=\frac{|x|}{|Z|}|X-Z|,
	\quad |y-z|=\frac{\rho^2|Y-Z|}{|Y||Z|}.
	\]
	Also
	\[
	V_\rho(y)^s=\left(\frac{|Y|}{\rho}\right)^{2ds}V(Y)^s,
	\quad
	V_\rho(z)^s=\left(\frac{|Z|}{\rho}\right)^{2ds}V(Z)^s.
	\]
	Substitution gives
	\begin{align*}
		& C|x|^{-\alpha}\iint
		\frac{V_\rho(y)^s|y|^{-\alpha}V_\rho(z)^s|z|^{-\alpha}}
		{|x-y|^\lambda |x-z|^\lambda |y-z|^\lambda}\,dy\,dz \\
		&=C|x|^{-\alpha-2\lambda}\rho^{-4ds-4\alpha+4n-2\lambda}
		\iint \frac{V(Y)^sV(Z)^s|Y|^{2ds+\alpha+2\lambda-2n}|Z|^{2ds+\alpha+2\lambda-2n}}
		{|X-Y|^\lambda |X-Z|^\lambda |Y-Z|^\lambda}\,dY\,dZ.
	\end{align*}
	Since $s=(n-d)/d$ and $d=\alpha+\lambda$, we have $2ds+\alpha+2\lambda-2n=-\alpha$. Thus the last display becomes
	\[
	C|x|^{-\alpha-2\lambda}\rho^{-4ds-4\alpha+4n-2\lambda}
	\iint \frac{V(Y)^s|Y|^{-\alpha}V(Z)^s|Z|^{-\alpha}}
	{|X-Y|^\lambda |X-Z|^\lambda |Y-Z|^\lambda}\,dY\,dZ.
	\]
	Since $V$ solves \eqref{eq:weighted-kelvin-eq} at $X$, the integral is $|X|^\alpha V(X)/C$. With $|X|=\rho^2/|x|$, this gives
	\[
	|x|^{-\alpha-2\lambda}\rho^{-4ds-4\alpha+4n-2\lambda}|X|^\alpha V(X)
	=\left(\frac\rho{|x|}\right)^{2(\alpha+\lambda)}V(X)=V_\rho(x).
	\]
\end{proof}

	The consequences of Proposition~\ref{prop:kelvin-weighted} differ
	in the unweighted and weighted cases. If \(\alpha=0\), then
	\(s=(n-\lambda)/\lambda\), and the scalar equation becomes
	\begin{equation}\label{eq:unweighted-scalar}
		V(x)=C\iint_{\Rn\times\Rn}
		\frac{V(y)^sV(z)^s}
		{|x-y|^\lambda |x-z|^\lambda |y-z|^\lambda}
		\,dy\,dz.
	\end{equation}
	This equation is translation invariant. Hence its origin-centered
	Kelvin invariance extends to every center: for \(x_0\in\Rn\) and
	\(\rho>0\), if \(V\) is a positive solution of
	\eqref{eq:unweighted-scalar}, then
	\[
	V_{x_0,\rho}(x):=
	\left(\frac{\rho}{|x-x_0|}\right)^{2\lambda}
	V\left(
	x_0+\frac{\rho^2(x-x_0)}{|x-x_0|^2}
	\right)
	\]
	is also a solution. Indeed, after translating \(x_0\) to the
	origin, this follows directly from
	Proposition~\ref{prop:kelvin-weighted} with \(\alpha=0\).
	
	If \(\alpha>0\), the weight \(|x|^{-\alpha}\) breaks translation
	invariance, so only the origin-centered conformal structure is
	retained. With conformal dimension \(d=\alpha+\lambda\), a Kelvin-compatible candidate profile is
	\[
	V_a(x)=A|x|^{-\alpha}
	\left(\frac{a}{a^2+|x|^2}\right)^\lambda.
	\]
	This ansatz is compatible with the origin-centered Kelvin transform, but Kelvin invariance alone does not show that it solves the weighted equation. Such a conclusion would require an additional weighted Selberg-type identity.
	
	Finally, the preceding arguments do not classify all
	maximizers or all positive finite-energy solutions. In the general
	weighted case, the compactness argument yields existence only.
	Even for \eqref{eq:unweighted-scalar}, a complete classification
	would require an additional argument, such as a moving-spheres
	theorem.
	
	\subsection{A conformal example at \texorpdfstring{$\lambda=n/3$}{lambda = n/3}}\label{subsec:conformal-example}
	
	We conclude with the explicit unweighted conformal case already contained \cite{Beckner1995,GrafakosMorpurgo1999}. Assume
	\[
	\alpha_1=\alpha_2=\alpha_3=0,
	\quad
	\alpha_{12}=\alpha_{13}=\alpha_{23}=\frac n3,
	\quad
	p_1=p_2=p_3=\frac32.
	\]
	Then
	\begin{equation}\label{eq:Q}
		Q(f_1,f_2,f_3)
		:=\iiint_{(\Rn)^3}
		\frac{f_1(x)f_2(y)f_3(z)}
		{|x-y|^{n/3}|x-z|^{n/3}|y-z|^{n/3}}\,dx\,dy\,dz.
	\end{equation}
	This is the case $k=3$ of Beckner's conformal multilinear fractional
	integral inequality \cite[Theorem~6]{Beckner1995}. Together with the
	three-fold Selberg integral formula of Grafakos--Morpurgo
	\cite[Corollary~1]{GrafakosMorpurgo1999}, that result yields
	\begin{equation}\label{eq:Qsharp}
		\abs{Q(f_1,f_2,f_3)}
		\le A_n\prod_{i=1}^3\norm{f_i}{3/2},
		\quad
		A_n=(2\pi)^n|\mathbb S^n|^{-1}
		\left(\frac{\Gamma(n/3)}{\Gamma(2n/3)}\right)^3.
	\end{equation}
	Here $|\mathbb S^n|$ denotes the surface measure of the unit sphere in
	$\mathbb R^{n+1}$. In particular, Beckner's result supplies the following
	conformal family of extremizing triples:
	\begin{equation}\label{eq:conformal-family-main}
		f_i(x)=c_i\left(\frac a{a^2+|x-x_0|^2}\right)^{2n/3},
		\quad c_i>0,
		\quad a>0,
		\quad x_0\in\Rn.
	\end{equation}
	
	For a nonzero diagonal member $f_1=f_2=f_3=f$ of this family, the
	Euler--Lagrange equation from Proposition~\ref{prop:EL} shows that, with
	$V=f^{1/2}$, there is a constant $\kappa>0$ such that
	\begin{equation}\label{eq:Vdiag}
		V(x)=\kappa\iint_{\Rn\times\Rn}
		\frac{V(y)^2V(z)^2}
		{|x-y|^{n/3}|x-z|^{n/3}|y-z|^{n/3}}\,dy\,dz
	\end{equation}
	for a.e. $x\in\Rn$. Consequently, the corresponding scalar profiles are
	\begin{equation}\label{eq:V-bubble}
		V(x)=A\left(\frac a{a^2+|x-x_0|^2}\right)^{n/3},
		\quad A>0,
		\quad a>0,
		\quad x_0\in\Rn.
	\end{equation}
	Since the right-hand side of \eqref{eq:Vdiag} is homogeneous of degree
	four in $V$, the amplitude $A$ may be adjusted so that the coefficient
	$\kappa$ is replaced by any prescribed positive constant. This example
	only records the conformal family furnished by the cited results; no
	classification of all extremizers or all finite-energy solutions is
	asserted here.

	\appendix
	\section{Five-singularity integral estimate}\label{app:five-singularity}
	
		\begin{proof}[Proof of Lemma~\ref{lem:five}]
		Let
		\[
		F(x,y):=
		\frac{1}
		{|x|^{\mu_1}|y|^{\mu_2}
			|x-e_1|^{\nu_1}|y-e_1|^{\nu_2}|x-y|^\lambda}.
		\]
		Throughout, $C>0$ denotes a constant independent of the dyadic
		indices. We first record the elementary local estimate
		\begin{equation}\label{eq:three-singularity-local-model}
			\iint_{B_1\times B_1}
			\frac{dx\,dy}{|x|^a|y|^b|x-y|^c}<\infty
		\end{equation}
		whenever
		\[
		0\le a,b,c<n,
		\qquad a+b+c<2n.
		\]
		Indeed, decompose $B_1\times B_1$, up to $(0,0)$, into
		\[
		E_m:=\bigl\{(x,y):2^{-m-1}<\max\{|x|,|y|\}\le2^{-m}\bigr\},
		\qquad m\ge0.
		\]
		After the scaling $(x,y)=2^{-m}(X,Y)$, the integral over $E_m$
		is bounded by $C2^{-m(2n-a-b-c)}$. The constant is finite because,
		on the fixed annulus
		\[
		\bigl\{(X,Y)\in B_1\times B_1:
		\tfrac12<\max\{|X|,|Y|\}\le1\bigr\},
		\]
		at least one of $|X|$ and $|Y|$ is bounded away from zero. For
		example, when $|X|>1/2$, split the $Y$-domain into $B_{1/4}$,
		$B_{1/4}(X)$, and the remaining region; the resulting integral is
		uniformly bounded because $b,c<n$. The case $|Y|>1/2$ is symmetric.
		Summing in $m$ proves \eqref{eq:three-singularity-local-model}.
		
		We next verify local integrability of $F$. Near $(0,0)$,
		\[
		F(x,y)\le C|x|^{-\mu_1}|y|^{-\mu_2}|x-y|^{-\lambda},
		\]
		so \eqref{eq:three-singularity-local-model} and \eqref{eq:five1}
		apply. After translating by $e_1$, the same argument near
		$(e_1,e_1)$ uses \eqref{eq:five2}. Near $(0,e_1)$ and $(e_1,0)$,
		$|x-y|$ is bounded away from zero and $F$ is dominated, respectively,
		by
		\[
		C|x|^{-\mu_1}|y-e_1|^{-\nu_2}
		\quad\text{and}\quad
		C|x-e_1|^{-\nu_1}|y|^{-\mu_2},
		\]
		which are locally integrable. At a point of the diagonal distinct
		from $(0,0)$ and $(e_1,e_1)$, only $|x-y|^{-\lambda}$ is singular;
		at every remaining point of the singular set, only one fixed-point
		weight is singular. Since all five individual exponents are less
		than $n$, it follows that $F\in L^1_{\mathrm{loc}}((\Rn)^2)$; hence
		\[
		\iint_{B_M\times B_M}F(x,y)\,dx\,dy<\infty
		\qquad(M>0).
		\]
		
		It remains to control infinity. Fix $R\ge8$ and set
		\[
		a_1:=\mu_1+\nu_1,
		\qquad a_2:=\mu_2+\nu_2,
		\qquad S:=a_1+a_2+\lambda.
		\]
		On $\{|x|\ge2R,\ |y|<R\}$, one has
		$|x-e_1|\simeq|x-y|\simeq|x|$, and hence
		\[
		F(x,y)\le
		C|x|^{-a_1-\lambda}|y|^{-\mu_2}|y-e_1|^{-\nu_2}.
		\]
		The $y$-factor is integrable on $B_R$, while the $x$-factor is
		integrable by \eqref{eq:five3}. Interchanging $x$ and $y$ gives the
		corresponding estimate on $\{|y|\ge2R,\ |x|<R\}$ by
		\eqref{eq:five4}. Since
		\[
		(\Rn)^2\subset (B_{2R}\times B_{2R})
		\cup\{|x|\ge2R,\ |y|<R\}
		\cup\{|y|\ge2R,\ |x|<R\}
		\cup\{|x|,|y|\ge R\},
		\]
		only the last region remains.
		
		For $k\ge0$, let
		\[
		r_k:=2^kR,
		\qquad D_k:=\{z\in\Rn:r_k\le|z|<2r_k\},
		\qquad
		I_{k,\ell}:=\iint_{D_k\times D_\ell}F(x,y)\,dx\,dy.
		\]
		Since $|z-e_1|\simeq|z|$ for $|z|\ge R$, if $k\ge\ell+2$ then
		$|x-y|\simeq|x|$ and
		\begin{equation}\label{eq:five-separated-shells}
			I_{k,\ell}
			\le C r_k^{n-a_1-\lambda}r_\ell^{n-a_2}.
		\end{equation}
		Put $\delta_1=n-a_1-\lambda<0$ and $\delta_2=n-a_2$.
		If $\delta_2<0$, the sum of $2^{\ell\delta_2}$ is uniformly
		bounded; if $\delta_2=0$, it contributes at most a factor $1+k$;
		and if $\delta_2>0$, it is bounded by $C2^{k\delta_2}$. In the
		last case,
		\[
		\delta_1+\delta_2=2n-S<0
		\]
		by \eqref{eq:five5}. Hence \eqref{eq:five-separated-shells} yields
		\[
		\sum_{k\ge2}\sum_{\ell=0}^{k-2}I_{k,\ell}<\infty.
		\]
		By symmetry,
		\[
		\sum_{\ell\ge2}\sum_{k=0}^{\ell-2}I_{k,\ell}<\infty,
		\]
		using \eqref{eq:five4} and \eqref{eq:five5}.
		
		Finally, suppose $|k-\ell|\le1$. With $x=r_kX$ and $y=r_kY$,
		the variables $X,Y$ range in fixed annuli bounded away from the
		origin. Since $r_k\ge R$, the quantities
		$|X-r_k^{-1}e_1|$ and $|Y-r_k^{-1}e_1|$ are also uniformly bounded
		above and below, and therefore
		\[
		I_{k,\ell}
		\le Cr_k^{2n-S}
		\iint_{\substack{1\le|X|<2\\[1pt]1/2\le|Y|<4}}
		|X-Y|^{-\lambda}\,dX\,dY
		\le Cr_k^{2n-S}.
		\]
		The last integral is finite because $\lambda<n$, and
		$2n-S<0$ by \eqref{eq:five5}. Since for each $k$ there are at most
		three such indices $\ell$, the comparable-scale sum also converges.
		Together with local integrability and the two one-variable far-field
		estimates, this proves \eqref{eq:five-int}.
	\end{proof}

	\backmatter

	\section*{Declarations}

	\bmhead{Funding}

	The authors declare that no funds, grants, or other support were received
	during the preparation of this manuscript.

	\bmhead{Competing interests}

	The authors have no relevant financial or non-financial interests to
	disclose.

	\bmhead{Ethics approval and consent to participate}

	Not applicable.

	\bmhead{Consent for publication}

	Not applicable.

	\bmhead{Data availability}

	This article is a theoretical study. No datasets were generated or analysed
	during the current study. All results and supporting arguments are contained
	in this article.

	\bmhead{Materials availability}

	Not applicable.

	\bmhead{Code availability}

	Not applicable.

	\bmhead{Author contributions}

	All authors contributed to the conception and development of the study and
	to the preparation of the manuscript. All authors read and approved the
	final manuscript.

	\bmhead{Originality statement}

	The authors declare that this manuscript presents original work, has not
	been published previously, and is not under consideration for publication
	elsewhere.

\end{document}